\documentclass{article}
\usepackage{amssymb}
\usepackage{eurosym}
\usepackage{epsfig}
\usepackage{amsmath}
\usepackage{amsfonts}
\usepackage{pgf,tikz}
\usepackage{enumerate}
\usepackage{mathrsfs}
\usepackage{cite}
\usepackage{subfigure}
\usepackage[colorlinks=true, linkcolor=black, citecolor=black, urlcolor=black]{hyperref}
\usepackage[top=2cm, bottom=2.5cm, left=2.5cm, right=2.5cm, headsep=1cm]{geometry}

\newtheorem{theorem}{\sc Theorem}[section]
\newtheorem{proposition}{\sc Proposition}[section]
\newtheorem{lemma}{\sc Lemma}[section]
\newtheorem{definition}{\sc Definition}[section]
\newtheorem{remark}{\sc Remark}[section]
\newtheorem{corollary}{\sc Corollary}[section]
\newtheorem{example}{\sc Example}[section]

\def\qed{\hbox to 0pt{}\hfill$\rlap{$\sqcap$}\sqcup$\medbreak}

\title{ Hybrid and component-wise Leggett-Williams type fixed point theorems in product spaces with applications}

\author{Laura Mª Fern\'andez--Pardo} 

\date{}
\begin{document}
 \maketitle

\begin{center}  {\small CITMAga \& Departamento de Estat\'{\i}stica, An\'alise Matem\'atica e Optimizaci\'on, \\ Universidade de Santiago de Compostela, \\ 15782, Facultade de Matem\'aticas, Campus Vida, Santiago, Spain.\\  Email: laura.fernandez.pardo@usc.es}
\end{center}

\medbreak

\noindent {\it Abstract.} In this paper, we present new multiplicity fixed point theorems for operators acting on Cartesian products of two normed linear spaces. We show that Leggett-Williams type conditions in each component of the system guarantee the existence of nine distinct fixed points, of which four of them are coexistence fixed points, i.e., points with all components nontrivial. In addition, a hybrid approach combining Leggett-Williams conditions in one component with Krasnosel'ski\u{\i} compression-expansion conditions in the other allows us to obtain three fixed points.
As an application, we establish the existence of multiple positive solutions for nonlinear systems of second-order equations with two-point boundary conditions.

\medbreak


\medbreak


\section{Introduction}
The existence and multiplicity of nontrivial solutions for various types of boundary value problems have long been central topics in nonlinear analysis. Among the most effective tools for establishing such results is fixed point theory in cones, with classical results such as the Krasnosel'ski\u{\i} compression-expansion theorem \cite{Kras} and the Leggett-Williams multiplicity theorem \cite{Leggett-Williams}. These results have played a key role in the localization of positive solutions, forming the foundation for many subsequent developments in the field.

In the case of systems, the fixed points guaranteed by both Krasnosel'ski\u{\i} and Leggett-Williams theorems are generally not localized independently in each component, leaving open the possibility that some components may be trivial—a limitation emphasized in \cite{JRL, LFP_JRL, PrecupFPT, PrecupSDC}. To address this issue, a vectorial version of the Krasnosel'ski\u{\i} theorem has been developed in \cite{PrecupFPT, PrecupSDC}, in which conditions are imposed independently on each component. This approach provides sufficient conditions for the existence of fixed points whose components are all nontrivial, referred to as \textit{coexistence fixed points} by Lan \cite{Lan1}. Other works that take advantage of the distinctive features arising from studying a system, rather than a single equation, to obtain results on the existence and localization of solutions include the following \cite{BCP,BGK23,CPR,CZ,InMaRo,INOP}.

It is clear that compression-expansion type results can produce directly multiplicity just by a nesting argument in different conical shells. However, by virtue of the additivity property of the fixed point index, additional solutions can be obtained, as established by the classical Leggett-Williams three-solutions theorem. The main objective of this work is to exploit the full strength of the aforementioned Leggett-Williams theorem to derive multiple solutions in the case of systems, with emphasis on coexistence fixed points. 

First, we establish a hybrid vectorial approach that combines, in a component-wise manner, the conditions of the Krasnosel'ski\u{\i} theorem with those of the Leggett-Williams theorem, yielding three fixed points, two of which are coexistence fixed points. This number of fixed points coincides with that ensured under the classical Leggett-Williams theorem; the difference, however, is that in the context of systems, our result guarantees the proper localization of each component, as mentioned. Second, we present a vectorial version of the classical Leggett-Williams multiplicity result, showing that for operators acting on the product of normed spaces, the number of fixed points grows exponentially with the number of components, highlighting the complexity of multiplicity phenomena in product spaces. In this way, nine fixed points are obtained (instead of three as in the classical setting!), four of which are coexistence fixed points. Our approach is based on the Leray-Schauder fixed point index in retracts.

In the original Leggett-Williams theorem, assumptions on the operator are imposed in sets of the form
\[
\overline{K}_c:=\{u\in K: \|u\|\leq c\}\quad \text{ and } \quad K(\alpha,a,b) := \{ u \in \overline{K}_c : \alpha(u) > a \text{ and } \|u\| \leq b \},
\]
where \(a,b,c \in \mathbb{R}_+ := (0,\infty)\) satisfy \(a < b\leq c\), \(K\) is a cone in the normed space \((X,\|\cdot\|)\) under consideration, and \(\alpha: K \to [0,\infty)\) is a continuous concave functional. The operator $T:\overline{K}_c\to K$ maps elements with $\alpha(u)=a$ outward in the sense that $\alpha(Tu)\ge a$ whereas the entire upper boundary of the domain $\|u\|=c$ ($c\in \mathbb{R}_+$, $b\le c$) is mapped inward, $\|Tu\|\le c$, yielding a compressive result. Applications to boundary value problems of this original result can be found, for instance, in \cite{guolak, kara, Xiaoming}.

In the early 2000s, the works \cite{AA, AH, ADH} extended this theorem by replacing norm-based bounds with convex functionals and by allowing subsets of both boundaries to be mapped inward and outward, thereby establishing an expansive version of the Leggett-Williams theorem and enabling a broader localization of multiple solutions. In this paper, we retain the original Leggett-Williams approach, preserving invariance on the upper boundary, while extending it to sets defined via concave and convex functionals in the line of those works. Specifically, we consider sets of the form
\[
\overline{\Omega}^\varphi_c:=\{u\in K: \varphi(u)\leq c\}\quad \text{ and } \quad \overline{\Omega}^\varphi_c(\alpha,\theta,a,b) := \{ u \in \overline{\Omega}^\varphi_c : \alpha(u) > a \text{ and } \theta(u) \leq b \},
\]
where \(\varphi, \theta:K\to [0,\infty)\) are convex functionals.

Our abstract setting is fairly general and can be applied to a variety of nonlinear systems involving integral or differential equations. We decided to illustrate it by extending to systems the problems studied in the original work of Leggett-Williams \cite{Leggett-Williams} for a single equation.
 
In this way, we investigate the existence of multiple solutions for the system
\begin{equation}
	\label{eq1}
	\left\{
	\begin{aligned}
		& u_1''(t) + f_1(u_1(t), u_2(t)) = 0, \quad t\in[0,1], \\
		& u_2''(t) + f_2(u_1(t), u_2(t)) = 0, \quad t\in[0,1],\\
		& u_1(0) = u_1'(1) = 0 = u_2(0) = u_2'(1),
	\end{aligned}
	\right.
\end{equation}
where the nonlinearities $f_1,f_2:[0,\infty)^2\to[0,\infty)$ are continuous. We first present a three-solutions result, by applying the hybrid Leggett-Williams and Krasnosel'ski\u{\i} approach (Theorem \ref{th_Kras_LW}) under expansive assumptions in the second component. 
It should be noted that, although the classical Leggett-Williams theorem can also guarantee the existence of three solutions for such a system, it does not ensure that at least two of them have both components nontrivial. A further key novelty of our hybrid approach resides in the fact that, whereas the classical theorem requires compressive behavior on the outer boundary, our theorem permits the second component to exhibit expansion, allowing the second nonlinearity to display different types of behavior while still ensuring the multiplicity result.

We also establish an existence result guaranteeing nine solutions, four of which possess nontrivial components, by means of the vectorial version of the Leggett-Williams theorem. This result is illustrated through the analysis of systems of the form \eqref{eq1} with symmetric nonlinearities, namely, those satisfying $f_2(u_1,u_2)=f_1(u_2,u_1)$ (see Example \ref{symmetric_systems} below). 

Finally, we establish the existence of four coexistence solutions for systems of the form
\begin{equation}
	\label{eq2}
	\left\{
	\begin{aligned}
		& \beta_1 u_1''(t) - u_1'(t) + f_1(u_1(t), u_2(t)) = 0, \quad t\in[0,1], \\
		& \beta_2 u_2''(t) - u_2'(t) + f_2(u_1(t), u_2(t)) = 0, \quad t\in[0,1],\\
		& \beta_1 u_1'(0) - u_1(0) = 0, \quad u_1'(1) = 0, \\
		& \beta_2 u_2'(0) - u_2(0) = 0, \quad u_2'(1) = 0,
	\end{aligned}
	\right.
\end{equation}
where $\beta_1,\beta_2>0$ and $f_1,f_2:[0,\infty)^2\to\mathbb{R}$ are continuous functions.  Systems of the form \eqref{eq2} naturally appear in reaction--convection--diffusion models and are relevant in applications such as chemical reactor theory and biological transport processes, see among others \cite{infante2,Leggett-Williams, Deim, varma}.

 The paper is organized as follows. In Section 2, we recall some basic properties of the Leray-Schauder fixed point index in cones and, moreover, we show specific computations of the index for operators defined in Cartesian products, which will be crucial in the proof of our main results. Sections 3 and 4 contain the hybrid and vectorial versions of Leggett-Williams theorem, respectively. Finally, Section 5 is concerned with the existence of multiple positive solutions for systems of second-order equations of the form \eqref{eq1} and \eqref{eq2}.

\section{ Preliminaries on fixed point index and computations}

\subsection{Key properties of the fixed point index}

At this point, we briefly recall some fundamental notions and the key properties of the fixed point index that will be used in our reasoning. Let $U$ be a relatively open bounded subset of a retract $R$ in a normed space $X$, and let $T: \overline{U} \to R$ be a compact operator that has no fixed points on the boundary of $U$ (denoted by $\partial_R U$). Under these conditions, the \emph{fixed point index} of $T$ on $U$ with respect to the retract $R$, denoted by $i_R(T,U)$, is well-defined. For further details, see \cite{amann,guolak}.

	\begin{proposition}
	Let $R$ be a retract of a normed space $X$, $U\subset R$ be a bounded relatively open set and $T:\overline{U}\rightarrow R$ be a compact map such that $T$ has no fixed points on $\partial_{R}\,U$. Then the fixed point index of $T$ on the set $U$ with respect to $R$, $i_{R}(T,U)$, has the following properties:
	\begin{enumerate}
		\item (Additivity) Let $U$ be the disjoint union of two open sets $U_1$ and $U_2$. If $0\not\in(I-T)(\overline{U}\setminus(U_1\cup U_2))$, then \[i_{R}(T,U)=i_{R}(T,U_1)+i_{R}(T,U_2).\]
		\item (Existence) If $i_{R}(T,U)\neq 0$, then there exists $x\in U$ such that $Tx=x$.
		\item (Homotopy invariance) If $H:\overline{U}\times[0,1]\rightarrow R$ is a compact homotopy and $0\not\in(I-H)(\partial_R\,U\times[0,1])$, then
		\[i_{R}(H(\cdot,0),U)=i_{R}(H(\cdot,1),U).\]
		\item (Normalization) If $T$ is a constant map with $Tx=\bar{x}$ for every $x\in\overline{U}$, then
		\[i_{R}(T,U)=\left\{\begin{array}{ll} 1, & \text{ if } \bar{x}\in U, \\ 0, & \text{ if } \bar{x}\not\in\overline{U}. \end{array} \right. \]
	\end{enumerate}
\end{proposition}

The following useful fixed point index computation may be derived from \cite[Lemma 2.4.1]{guolak}.

\begin{proposition}
	\label{prop_index}
	Let $R$ be a retract of a normed space $X$ and $R_1$ a bounded retract of $R$.  Let $U$ be a nonempty relatively open subset of $R$ such that $U\subset R_1$. Suppose that $T:R_1\rightarrow R$ is completely continuous, $T(R_1)\subset R_1$ and $T$ has no fixed points on $R_1\setminus U$. Then $i_R(T,U)=1$.
\end{proposition}

Concerning the fixed point index, it is common to work in the context of cones instead of general retracts.

\begin{definition}
 A \emph{cone} in a normed linear space $X$ is a closed convex subset $K$ satisfying $\lambda u \in K$ for every $u \in K$ and all $\lambda \geq 0$, and such that $K \cap (-K) = \{0\}$.
\end{definition}

It is well-known that a cone is a retract due to its closedness and convexity \cite[Corollary 4.2]{Dujun}. 
 
\subsection{Main fixed point index computations in product spaces}

In the sequel, let $(X,\|\cdot\|_1)$ and $(Y,\|\cdot\|_2)$ be normed linear spaces, and let $K_1 \subset X$ and $K_2 \subset Y$ be cones. We set $K := K_1 \times K_2$, which is itself a cone in $X \times Y$. When no ambiguity is expected, both norms will be denoted simply by $\|\cdot\|$.

Let $\alpha,\beta: K_1 \rightarrow [0,\infty)$ be continuous concave and convex functionals on $K_1$, respectively; that is, continuous functions satisfying
\[\begin{aligned}
	& \alpha(\lambda u + (1-\lambda)v) \ge \lambda \alpha(u) + (1-\lambda)\alpha(v), 
	\quad \text{for all } u, v \in K_1 \text{ and } \lambda \in [0,1],\, \text{ and }\\
	& \beta(\lambda u + (1-\lambda)v) \le \lambda \beta(u) + (1-\lambda)\beta(v), 
	\quad \text{for all } u, v \in K_1 \text{ and } \lambda \in [0,1].
\end{aligned}\]
In addition, let $\varphi : K_1 \rightarrow [0,\infty)$ be a continuous convex functional and denote for $a\in\mathbb{R}_+$
\[\begin{aligned}
	& \Omega^{\varphi}_a:=\{u\in K_1: \varphi(u)< a\},\\
	& \overline{\Omega}^{\varphi}_a:=\{u\in K_1: \varphi(u)\leq a\}.
\end{aligned}\]

For $a,b,c\in\mathbb{R}_+$ and arbitrary continuous functionals $\gamma,\theta,\psi:K_1\rightarrow[0,\infty)$, consider the following notation
\[\begin{aligned}
	& K_1(\gamma,\theta,a,b)=\{u\in K_1: a<\gamma(u), \, \theta(u)\leq b\},\\
	& K_1[\gamma,\theta,a,b]=\{u\in K_1: a\leq\gamma(u), \, \theta(u)\leq b\},\\
	& \overline{\Omega}_c^{\psi}(\gamma,\theta,a,b)=\{u\in \overline{\Omega}_c^{\psi}: a<\gamma(u), \, \theta(u)< b\},\\
	& \overline{\Omega}_c^{\psi}[\gamma,\theta,a,b]=\{u\in \overline{\Omega}_c^{\psi}: a\le\gamma(u), \, \theta(u)\le b\},\\
	& \overline{\Omega}_c^{\psi}(\gamma,a)=\{u\in \overline{\Omega}_c^{\psi}:  \gamma(u)< a\},\\
	& \overline{\Omega}_c^{\psi}[\gamma,a]=\{u\in \overline{\Omega}_c^{\psi}:  \gamma(u)\le a\}.
\end{aligned}\]

We introduce here the first fixed point index computation, obtained by combining a Leggett-Williams type condition in one component with the Leray-Schauder condition in the other.

\begin{proposition}
	\label{prop1}
	Let $a,b,c\in \mathbb{R}_+$ be such that $a<b\leq c$, $U\subset K_2$ a bounded relatively open set with $0\in U$. Suppose that $T:\overline{\Omega}^\varphi_c\times\overline{U}\rightarrow K$ is a compact operator satisfying the following conditions
		\begin{enumerate}[(A)]
			\item
			\begin{enumerate}[(a)]
				\item $\{u\in \overline{\Omega}_c^{\varphi}[\alpha,\beta,a,b]\times \overline{U}: \alpha(u_1)>a\}\neq \emptyset$ and $\alpha(T_1u)>a$ for $u\in \overline{\Omega}_c^{\varphi}[\alpha,\beta,a,b]\times U$ with $\alpha(u_1)=a$;
				\item $\varphi(T_1 u)\leq c$ for $u\in K_1[\alpha,\varphi,a,c]\times \overline{U}$;
				\item $\alpha(T_1 u)>a$ for $u\in K_1[\alpha,\varphi,a,c]\times U$ with $\alpha(u_1)=a$ and $\beta(T_1u)>b$,
			\end{enumerate}
			\item $T_2u\neq \lambda u_2$ for $u_1\in K_1[\alpha,\varphi,a,c]$, $u_2\in \partial_{K_2}U$ and $\lambda\geq 1$.
	\end{enumerate}
	Then 
	\[i_{\overline{\Omega}^\varphi_c\times K_2}\left(T,K_1(\alpha,\varphi,a,c)\times U\right)=1.\]
\end{proposition}

\noindent
{\bf Proof.} First, we show that \(T\) has no fixed points on 
\[
\partial_{\overline{\Omega}^\varphi_c\times K_2}
\bigl(K_1[\alpha,\varphi,a,c]\times \overline{U}\bigr)
=
\bigl(\{u_1\in \overline{\Omega}_c^{\varphi}:\alpha(u_1)=a\}\times U\bigr)
\cup
\bigl(K_1[\alpha,\varphi,a,c]\times \partial_{K_2} U\bigr),
\]
which together with $T_1\left(K_1[\alpha,\varphi,a,c]\times \overline{U}\right)\subset\overline{\Omega}_c^\varphi$ as a consequence of $(A)\text{--}(b)$, ensures that the corresponding fixed point index is well defined.

Assume, to the contrary, that there exists 
\(u\in \{u_1\in \overline{\Omega}_c^{\varphi}:\alpha(u_1)=a\}\times U\) 
such that \(Tu=u\). Then \(T_1 u=u_1\), with \(\alpha(u_1)=a\) and 
\(u_2\in U\). If \(u_1\in \overline{\Omega}_c^{\varphi}[\alpha,\beta,a,b]\), then 
\(\alpha(T_1 u)=\alpha(u_1)=a\), contradicting \((A)\text{--}(a)\).
Alternatively, if \(\beta(u_1)>b\), then \(\beta(T_1u)>b\). 
Thus, by \((A)\text{--}(c)\), we obtain \(\alpha(T_1u)>a\), again a contradiction. Moreover, there is clearly no 
\(u\in K_1[\alpha,\varphi,a,c]\times \partial_{K_2} U\) 
such that \(Tu=u\). Indeed, this would imply \(T_2u=u_2\), which is excluded by 
condition~(B) with \(\lambda=1\).

Now, for a fixed
\(\overline{u} \in \overline{\Omega}_c^{\varphi}[\alpha,\beta,a,b]\times \overline{U}\) satisfying 
\(\alpha(\overline{u}_1)>a\), consider the map  
\[
h:K_1[\alpha,\varphi,a,c]\times \overline{U}\times[0,1] \to \overline{\Omega}_c^\varphi\times K_2,
\qquad
h(u_1,u_2,t)=\bigl(t T_1u+(1-t)\overline{u}_1,\; t T_2 u\bigr).
\]
Since $T$ is compact, $h$ is a compact homotopy. Moreover, it contains no points  
\[
(u,t)\in \partial_{\overline{\Omega}^{\varphi}_{c}\times K_2}
\bigl(K_1[\alpha,\varphi,a,c]\times \overline{U}\bigr)\times [0,1]
\]
such that $h(u,t)=u$. 

Indeed, assume that such a point exists with 
$\alpha(u_1) = a$ and $u_2 \in U$. If $\beta(T_1 u) > b$, then \((A)\text{--}(c)\) implies $\alpha(T_1 u) > a$. Therefore, $\alpha(u_1) = \alpha\bigl(t T_1 u + (1-t)\overline{u}_1\bigr)
\ge t\,\alpha(T_1 u) + (1-t)\,\alpha(\overline{u}_1) > a,$ a contradiction. If, on the other hand $\beta(T_1 u) \le b$, then $\beta(u_1) = \beta\bigl(t T_1 u + (1-t)\overline{u}_1\bigr)
\le t\,\beta(T_1 u) + (1-t)\,\beta(\overline{u}_1) \le b$. Hence, $u_1 \in \overline{\Omega}_c^{\varphi}[\alpha,\beta,a,b]$, implying $\alpha(T_1 u) > a$ by \((A)\text{--}(a)\). Consequently, $\alpha(u_1) = \alpha\bigl(t T_1 u + (1-t)\overline{u}_1\bigr) > a,$ which is impossible.
Otherwise, if the point is such that $u_1 \in K_1[\alpha,\varphi,a,c]$
and $u_2 \in \partial_{K_2}\, U$, then we would have $t T_2 u = u_2$. 
For $t = 0$, this gives $u_2 = 0 \in  \partial_{K_2}\,U$, which is not possible since $0\in U$ and $U$ is open. 
For $t \in (0,1]$, we obtain $T_2 u = \frac{1}{t} u_2$, also leading to a contradiction, this time with $(B)$.

Therefore, by the homotopy invariance and normalization properties of the fixed point index, we obtain
\[
i_{\overline{\Omega}^\varphi_c\times K_2}
\bigl(T,\; K_1(\alpha,\varphi,a,c)\times U\bigr)
=
i_{\overline{\Omega}^\varphi_c\times K_2}
\bigl((\overline{u}_1,0),\; K_1(\alpha,\varphi,a,c)\times U\bigr)
= 1.
\]
\qed

\begin{proposition}
	\label{prop2}
	Let $a,b,c\in \mathbb{R}_+$ be such that $a<b\leq c$ and $U\subset K_2$ a bounded relatively open set. Suppose that $T:\overline{\Omega}^\varphi_c\times\overline{U}\rightarrow K$ and $S:K_1[\alpha,\varphi,a,c]\times \overline{U}\rightarrow K_2$ are compact operators satisfying the following conditions
	\begin{enumerate}[(A)]
		\item
		\begin{enumerate}[(a)]
			\item $\{u\in \overline{\Omega}_c^{\varphi}[\alpha,\beta,a,b]\times \overline{U}: \alpha(u_1)>a\}\neq \emptyset$ and $\alpha(T_1u)>a$ for $u\in \overline{\Omega}_c^{\varphi}[\alpha,\beta,a,b]\times U$ with $\alpha(u_1)=a$;
			\item $\varphi(T_1 u)\leq c$ for $u\in K_1[\alpha,\varphi,a,c]\times \overline{U}$;
			\item $\alpha(T_1 u)>a$ for $u\in K_1[\alpha,\varphi,a,c]\times U$ with $\alpha(u_1)=a$ and $\beta(T_1u)>b$,
		\end{enumerate}
		\item 
		\begin{enumerate}
			\item $\inf_{u\in K_1[\alpha,\varphi,a,c]\times \overline{U}}\|Su\|>0$;
			\item $x_2-T_2 u\neq \mu S u$ for $u_1\in K_1[\alpha,\varphi,a,c]$, $u_2\in\partial_{K_2} U$ and $\mu\geq 0$.
		\end{enumerate}
	\end{enumerate}
	Then
	\[i_{\overline{\Omega}^\varphi_c\times K_2}\left(T,K_1(\alpha,\varphi,a,c)\times U\right)=0.\]
\end{proposition}

\noindent
{\bf Proof.} As in the previous result, \(T\) has no fixed points on the boundary of
\(K_1[\alpha,\varphi,a,c]\times \overline{U}\). Therefore, the fixed point index is well defined.

Consider the following positive numbers
\[
\alpha := \displaystyle \sup_{u_2 \in \overline{U}} \|u_2\|, \quad
\beta := \displaystyle \sup_{u \in K_1[\alpha,\varphi,a,c] \times \overline{U}} \|T_2 u\|, \quad
\gamma := \displaystyle \inf_{u \in K_1[\alpha,\varphi,a,c] \times \overline{U}} \|Su\|.
\]

Choose $\overline{u} \in \overline{\Omega}_c^{\varphi}[\alpha,\beta,a,b]\times\overline{U}$ satisfying $\alpha(\overline{u}_1) > a$ as well as
$\mu_0 > (\alpha + \beta)/\gamma$. We introduce the mapping
\[
h :  K_1[\alpha,\varphi,a,c] \times \overline{U} \times[0,1]\to \overline{\Omega}_c^\varphi\times K_2, 
\qquad
h(u_1,u_2,t) = \bigl(t T_1 u + (1-t) \overline{u}_1,\; T_2 u + (1-t)\mu_0 S u\bigr).
\]

Again, $h$ is an admissible compact homotopy. That is, there exists no points $(u,t)$ such that $u\in \partial_{\overline{\Omega}^{\varphi}_{c} \times K_2} 
\bigl(K_1[\alpha,\varphi,a,c]\times \overline{U}\bigr)$ and $t\in[0,1]$ satisfying $h(u,t) = u$. Arguing as in the previous proof, there are no such points with 
$\alpha(u_1) = a$ and $u_2 \in U$. 
Similarly, there are no such points satisfying 
$u_1 \in K_1[\alpha,\varphi,a,c]$ and $u_2 \in \partial_{K_2} U$, 
because in this case we would have $u_2 = T_2 u + (1-t)\mu_0 S u,$ which contradicts condition $(B)$.

Applying then the homotopy invariance property of the fixed point index, we have
\[
i_{\overline{\Omega}^\varphi_c \times K_2}
\bigl(T,\; K_1(\alpha,\varphi,a,c) \times U\bigr)
=
i_{\overline{\Omega}^\varphi_c \times K_2}
\bigl(H(\cdot,0),\; K_1(\alpha,\varphi,a,c) \times U\bigr).
\]

Moreover, $H(\cdot,0)$ has no fixed points in $K_1(\alpha,\varphi,a,c) \times U$. 
Indeed, suppose that there exists $(u_1,u_2)$ such that $u_1 = \overline{u}_1$ and 
$u_2 = T_2(\overline{u}_1,u_2) + \mu_0 S(\overline{u}_1,u_2) $. Then
\[
\mu_0 = \frac{\|u_2 - T_2(\overline{u}_1,u_2)\|}{\|S(\overline{u}_1,u_2)\|} 
\le \frac{\alpha + \beta}{\gamma},
\]
which yields a contradiction. 

The conclusion now follows from the existence property of the fixed point index.
 \qed
 
 \begin{remark}
 	\label{remark1}
 	Note that, in the previous proposition, if $\partial_{K_2} U$ is a retract of $\overline{U}$, we can rewrite the hypotheses by defining $S$ over the set $K_1[\alpha,\varphi,a,c]\times \partial_{K_2}U$ and changing $(B)\text{--}(a)$ by
 	\[\inf_{u\in K_1[\alpha,\varphi,a,c]\times \partial_{K_2}U}\|Su\|>0.\]
 \end{remark}
 
 \begin{proposition}
 	\label{prop3}
Let $c,d\in \mathbb{R}_+$ be such that $d\leq c$ and $\varphi,\psi:K_1\rightarrow[0,+\infty)$ continuous convex funtionals. Additionally, let $U\subset K_2$ be a relatively open bounded subset of $K_2$ and $T:\overline{\Omega}^\varphi_c[\psi,d]\times \overline{U}\rightarrow\overline{\Omega}^\varphi_c\times K_2$ a compact operator such that $T_1(\overline{\Omega}^\varphi_c[\psi,d]\times U)\subset\overline{\Omega}^\varphi_c(\psi,d)$.
\begin{enumerate}[(i)]
	\item If $0\in U$ and
	\begin{equation}
		\label{ec1}
		T_2 u\neq \lambda u_2, \text{ for } u_1\in\overline{\Omega}^\varphi_c[\psi,d], \, u_2\in \partial_{K_2}U \text{ and }\lambda\ge1,
	\end{equation}
	then 
	\[i_{\overline{\Omega}^\varphi_c\times K_2}\left(T,\overline{\Omega}^\varphi_c(\psi,d)\times U\right)=1.\]
\item If there exists $S:\overline{\Omega}^\varphi_c[\psi,d]\times \overline{U}\rightarrow K_2$ a compact operator such that
	\begin{enumerate}[(a)]
	\item $\inf_{u\in \overline{\Omega}^\varphi_c[\psi,d]\times \overline{U}}\|Su\|>0$;
	\item $u_2-T_2 u\neq \mu S u$ for $u_1\in \overline{\Omega}^\varphi_c[\psi,d]$, $u_2\in\partial_{K_2} U$ and $\mu\geq 0$,
\end{enumerate}
then
	\[i_{\overline{\Omega}^\varphi_c\times K_2}\left(T,\overline{\Omega}^\varphi_c(\psi,d)\times U\right)=0.\]
\end{enumerate}
 \end{proposition}
 
 \noindent
 {\bf Proof.} Note that, since $T_1(\overline{\Omega}^\varphi_c[\psi,d] \times U) \subset \overline{\Omega}^\varphi_c(\psi,d)$, we have
 \begin{equation}
 	\label{ec2}
 	T_1 u \neq \lambda u_1 \quad \text{for all } u \in \overline{\Omega}^\varphi_c \times U \text{ with } \psi(u_1) = d \text{ and } \lambda \ge 1.
 \end{equation}
 Indeed, if the contrary were true, there would exist $u \in \overline{\Omega}^\varphi_c[\psi,d] \times U$ and $\lambda \ge 1$ such that $\psi(u_1) = d$ and $T_1 u = \lambda u_1$. Then, using that $\psi$ is convex, we would obtain
 \[
 d = \psi(u_1) = \psi\left(\frac{1}{\lambda} T_1 u \right) \leq \frac{1}{\lambda} \psi(T_1 u),
 \]
 which implies $\psi(T_1 u) \geq \lambda d \geq d$, contradicting the fact that $T_1 u \in \overline{\Omega}^\varphi_c(\psi,d)$.
 
 For the proof of $(i)$, we introduce the mapping
 \[
 H_{(i)}: \overline{\Omega}^\varphi_c[\psi,d]\times \overline{U} \times [0,1] \longrightarrow \overline{\Omega}^\varphi_c \times K_2, 
 \quad H_{(i)}(u,t) = t \, T u.
 \]
 This defines a compact homotopy such that $H_{(i)}(u,t) \neq u$ for all $u \in \partial_{\overline{\Omega}^\varphi_c \times K_2}(\overline{\Omega}^\varphi_c[\psi,d]\times \overline{U})$ and $t \in [0,1]$. Suppose, by contradiction, that there exists $u \in (\{u_1\in\overline{\Omega}_c^\varphi:\psi(u_1)=d\} \times U) \cup (\overline{\Omega}^\varphi_c[\psi,d]\times \partial_{K_2} U)$ and $t \in [0,1]$ such that $u = t \, T u$. If $t = 0$, then $u = 0$, which is impossible since both $\overline{\Omega}^\varphi_c(\psi,d)$ and $U$ are relatively open sets containing the origin. If instead $t \in (0,1]$, then $T u = \frac{1}{t} u$, which leads to a contradiction with \eqref{ec1} if $u \in\overline{\Omega}^\varphi_c[\psi,d] \times \partial_{K_2} U$, or with \eqref{ec2} if $u \in\{u_1\in\overline{\Omega}_c^\varphi:\psi(u_1)=d\}  \times U$. Consequently, by the homotopy invariance and normalization properties of the fixed point index, we obtain 
 \[
 i_{\overline{\Omega}^\varphi_c \times K_2} \left(T,\overline{\Omega}^\varphi_c(\psi,d)\times U\right)
 = i_{\overline{\Omega}^\varphi_c \times K_2} \left(H_{(i)}(\cdot, 0), \overline{\Omega}^\varphi_c(\psi,d)\times U\right)
 = 1.
 \]
 
 Concerning $(ii)$, an analogous reasoning applies. Let
 \[\alpha := \displaystyle \sup_{u_2 \in \overline{U}} \|u_2\|, \quad
 \beta := \displaystyle \sup_{u \in\overline{\Omega}^\varphi_c[\psi,d] \times \overline{U}} \|T_2 u\|, \quad
 \gamma := \displaystyle \inf_{u \in \overline{\Omega}^\varphi_c[\psi,d] \times \overline{U}} \|Su\|.\]
 
  Consider the compact homotopy
 \[
 H_{(ii)}: \overline{\Omega}^\varphi_c[\psi,d] \times \overline{U} \times [0,1] \longrightarrow \overline{\Omega}^\varphi_c \times K_2, 
 \quad H_{(ii)}(u,t) = \bigl(t T_1 u,\, T_2 u + (1-t) \mu_0 S u\bigr),
 \]
 where $\mu_0 > (\alpha + \beta)/\gamma$. It is straightforward to verify, by following the proof of Proposition \ref{prop2}, that $H_{(ii)}$ is admissible and $H_{(ii)}(\cdot,0)$ has no fixed points in $\overline{\Omega}^\varphi_c(\psi,d)\times U$. Therefore, the existence property together with the homotopy invariance of the fixed point index yields
 \[
 i_{\overline{\Omega}^\varphi_c \times K_2} \left(T,\overline{\Omega}^\varphi_c(\psi,d) \times U\right)
 = i_{\overline{\Omega}^\varphi_c \times K_2} \left(H_{(ii)}(\cdot,0), \overline{\Omega}^\varphi_c(\psi,d)\times U\right)
 = 0.
 \]
 \qed
 
 \begin{remark}
Note that in Proposition~\ref{prop3} the condition
$T_1(\overline{\Omega}^\varphi_c[\psi,d]\times U) \subset \overline{\Omega}^\varphi_c(\psi,d)$
can be weakened by assuming \eqref{ec2}. Nevertheless, in the subsequent sections we shall apply the proposition under its original hypotheses, as stated.

Moreover, when $c = d$ and $\psi = \varphi$, the condition $T_1(\overline{\Omega}^\varphi_c[\psi,d] \times U) \subset \overline{\Omega}^\varphi_c(\psi,d)$, which is equivalent to $T_1(\overline{\Omega}^\varphi_c \times U) \subset \Omega^\varphi_c$ is not necessary. Since we are computing the fixed point index with respect to $\overline{\Omega}^\varphi_c$ in the first component, it suffices to assume that $T_1(\overline{\Omega}^\varphi_c \times U) \subset \overline{\Omega}^\varphi_c$.

 \end{remark}
 
 \section{Multiple fixed points in product spaces via a combination of Krasnosel’ski\u{\i} and Leggett-Williams conditions}
 
Our first result in this section provides sufficient conditions ensuring the existence of at least one fixed point with all components nontrivial, via a computation of the fixed point index. It will play a key role in establishing our first multiplicity theorem, where conditions combining Krasnosel'ski\u{\i}-type assumptions on one component with Leggett-Williams conditions on the other guarantee the existence of at least three fixed points, two of which have both components nontrivial.

Throughout this section, we work under the notation and assumptions introduced in the previous section.
In addition, for a given $r\in\mathbb{R}_+$, we introduce the following notation
\[\begin{aligned}
	&\overline{\Omega}_r^{\|\cdot\|_2}:=\{u\in K_2: \, \|u\|\leq r\},\\
	& \Omega_r^{\|\cdot\|_2}:=\{u\in K_2: \, \|u\|< r\}.
\end{aligned}\]
More generally, from now on, given a functional $\eta: K_2 \to \mathbb{R}$, we will denote 
\[\begin{aligned}
	&\overline{\Omega}_r^{\eta}:=\{u\in K_2: \, \eta(u)\leq r\},\\
	& \Omega_r^{\eta }:=\{u\in K_2: \, \eta(u)< r\}.
\end{aligned}\]
Observe that we previously used a similar notation for operators on $K_1 \subset X$. Here, we specify $\|\cdot\|_2$ in $\overline{\Omega}_r^{\|\cdot\|_2}$ and $\Omega_r^{\|\cdot\|_2}$ to indicate that the functional is defined on $K_2 \subset Y$, avoiding ambiguity.

Let $r,R\in \mathbb{R}_+$ with $r<R$ and consider also the following sets
\[\begin{aligned}
&(\overline{K}_2)_{r,R}:=\{u\in K_2: \, r\leq \|u\|\leq R\},\\
& (K_2)_{r,R}:=\{u\in K_2: \, r< \|u\|< R\}.
\end{aligned}\]
We note that $(\overline{K}_2)_{r,R}$ is a retract of $\overline{\Omega}^{\|\cdot\|_2}_R$ (see \cite{JRL, LFP_JRL} for an explicit construction), which will be used in the following fixed point result.

\begin{proposition}
	\label{Prop_Kras_LW}
	Let $a,b,c,r,R\in \mathbb{R}_+$ be such that $a<b\leq c$ and $r<R$. Suppose that $T:\overline{\Omega}^{\varphi}_c\times(\overline{K}_2)_{r,R}\rightarrow K$ is a compact operator satisfying the following conditions
	\begin{enumerate}[(A)]
		\item
		\begin{enumerate}[(a)]
			\item $\{u\in \overline{\Omega}_c^\varphi[\alpha,\beta,a,b]\times (\overline{K}_2)_{r,R}: \alpha(u_1)>a\}\neq \emptyset$ and $\alpha(T_1u)>a$ for $u\in \overline{\Omega}_c^\varphi[\alpha,\beta,a,b]\times (\overline{K}_2)_{r,R}$ with $\alpha(u_1)=a$;
			\item $\varphi(T_1 u)\leq c$ for $u\in K_1[\alpha,\varphi,a,c]\times (\overline{K}_2)_{r,R}$;
			\item $\alpha(T_1 u)>a$ for $u\in K_1[\alpha,\varphi,a,c]\times (\overline{K}_2)_{r,R}$ with $\alpha(u_1)=a$ and $\beta(T_1u)>b$,
		\end{enumerate}
		\item For $u\in K_1[\alpha,\varphi,a,c]\times(\overline{K}_2)_{r,R}$
		\begin{enumerate}[(i)]
			\item 	$\|T_2 u\|>\|u_2\|$ if $\|u_2\|=r$ and $\|T_2 u\|<\|u_2\|$ if $\|u_2\|=R$; or
			\item	$\|T_2 u\|<\|u_2\|$ if $\|u_2\|=r$ and $\|T_2 u\|>\|u_2\|$ if $\|u_2\|=R$.
		\end{enumerate}
	\end{enumerate}
Then
\[
i_{\overline{\Omega}^\varphi_c \times K_2}\bigl(T, K_1(\alpha,\varphi,a,c) \times (K_2)_{r,R}\bigr)
= (-1)^s,
\qquad
s =\left\{\begin{array}{ll} 0, & \quad \text{if $(B)\text{--}(i)$ holds}, \\ 	1, & \quad \text{if $(B)\text{--}(ii)$ holds}. \end{array} \right. 
\]
In either case, $T$ admits a fixed point $(u_1,u_2)$ in $K_1(\alpha,\varphi,a,c) \times (K_2)_{r,R}$.
\end{proposition}

\noindent
{\bf Proof.} Assume that condition $(B)$--$(ii)$ holds. Then the second component of the operator exhibits expansive behavior. The case $(B)$--$(i)$, in which the second component is compressive, can be treated analogously.

As a first step, we extend the operator $T$ to
$\overline{\Omega}^{\varphi}_c\times\overline{\Omega}^{\|\cdot\|_2}_R$ by defining
\[
\tilde{T}:
\overline{\Omega}^{\varphi}_c\times\overline{\Omega}^{\|\cdot\|_2}_R
\longrightarrow K,
\qquad
\tilde{T}(u_1,u_2)
:=
T\bigl(u_1, \rho_2(u_2)\bigr),
\]
where $\rho_2 : \overline{\Omega}^{\|\cdot\|_2}_R \to (\overline{K}_2)_{r,R}$ is a retraction of $\overline{\Omega}^{\|\cdot\|_2}_R$ onto $(\overline{K}_2)_{r,R}$.

Observe that $\tilde{T}$ satisfies condition $(A)$ in
Propositions~\ref{prop1} and~\ref{prop2} by taking
$\overline{U}:=\overline{\Omega}^{\|\cdot\|_2}_R$ in those statements.

Condition $(B)$--$(ii)$ implies that 
\[
\tilde{T}_2u\neq \lambda u_2
\quad
\text{for }
u_1\in K_1[\alpha,\varphi,a,c],\ \|u_2\|=r \text{ and } \lambda\ge1.
\]
Indeed, otherwise there would exist
$\bar u=(\bar u_1,\bar u_2)\in
K_1[\alpha,\varphi,a,c]\times K_2$
with $\|\bar u_2\|=r$ and $\bar\lambda\ge1$ such that
$\tilde{T}_2\bar u=\bar\lambda\,\bar u_2$.
Consequently,
$\|\tilde{T}_2\bar u\|=\bar\lambda\|\bar u_2\|\ge \|\bar u_2\|$,
which contradicts $(B)$--$(ii)$ because $\tilde{T}_2\bar u=T_2\bar u$.
Proposition~\ref{prop1} therefore yields
\[
i_{\overline{\Omega}^\varphi_c\times K_2}
\left(\tilde T,K_1(\alpha,\varphi,a,c)\times \Omega^{\|\cdot\|_2}_r\right)=1.
\]

On the other hand, since
$\|\tilde{T}_2u\|>\|u_2\|$ for
$u_1\in K_1[\alpha,\varphi,a,c]$ and $\|u_2\|=R$,
we have
\[
\tilde{T}_2u\neq \mu u_2
\quad
\text{for }
u_1\in K_1[\alpha,\varphi,a,c],\ \|u_2\|=R\text{ and } \mu\in(0,1].
\]
This implies that
\[
u_2-\tilde{T}_2u\neq \mu \tilde{T}_2u
\quad
\text{for }
u_1\in K_1[\alpha,\varphi,a,c],\ \|u_2\|=R \text{ and } \mu\ge0.
\]
Indeed, if this were not the case, there would exist
$\bar u=(\bar u_1,\bar u_2)$ with
$\bar u_1\in K_1[\alpha,\varphi,a,c]$,
$\|\bar u_2\|=R$, and $\bar\mu\ge0$ such that
$\bar u_2-\tilde{T}_2\bar u=\bar\mu\,\tilde{T}_2\bar u$.
Hence,
$\tilde{T}_2\bar u=(1+\bar\mu)^{-1}\bar u_2$ with
$(1+\bar\mu)^{-1}\in(0,1]$, a contradiction.

Moreover,
\[
\inf_{u\in
	K_1[\alpha,\varphi,a,c]\times
	\partial_{K_2}\Omega^{\|\cdot\|_2}_R}
\|\tilde{T}_2u\|
=
\inf_{u\in
	K_1[\alpha,\varphi,a,c]\times
	\partial_{K_2}\Omega^{\|\cdot\|_2}_R}
\|T_2u\|
>R.
\]
Taking Remark~\ref{remark1} into account and applying
Proposition~\ref{prop2} with
$S:=\tilde{T}$ and $\overline{U}:=\overline{\Omega}^{\|\cdot\|_2}_R$, we obtain
\[
i_{\overline{\Omega}^\varphi_c\times K_2}
\left(\tilde T,K_1(\alpha,\varphi,a,c)\times \Omega^{\|\cdot\|_2}_R\right)=0.
\]

Since $\tilde{T}$ has no fixed points $u = (u_1,u_2)$ with
$u_1 \in K_1[\alpha,\varphi,a,c]$ and $\|u_2\| = r$,
we can apply the additivity property of the fixed point index, which yields
\[
\begin{aligned}
	i_{\overline{\Omega}^\varphi_c \times K_2}
	\bigl(\tilde{T}, K_1(\alpha,\varphi,a,c)\times (K_2)_{r,R}\bigr)
	&=
	i_{\overline{\Omega}^\varphi_c \times K_2}
	\left(\tilde{T}, K_1(\alpha,\varphi,a,c)\times \Omega^{\|\cdot\|_2}_R\right)
	\\
	&\quad -
	i_{\overline{\Omega}^\varphi_c \times K_2}
	\left(\tilde{T}, K_1(\alpha,\varphi,a,c)\times \Omega^{\|\cdot\|_2}_r\right)
	= -1.
\end{aligned}
\]
Finally, since $\tilde{T}=T$ on
$K_1[\alpha,\varphi,a,c]\times (\overline{K}_2)_{r,R}$,
it follows that
\[
i_{\overline{\Omega}^\varphi_c \times K_2}
\bigl(T, K_1(\alpha,\varphi,a,c)\times (K_2)_{r,R}\bigr)=-1.
\]
The existence of the fixed point is now a consequence of the existence property of the fixed point index.
\qed

\begin{theorem} 
	\label{th_Kras_LW}
	\emph{(Leggett-Williams $\times$ Krasnosel'ski\u{\i})}
	Let $a,b,c,d,r,R\in \mathbb{R}_+$ be such that $d<a<b\leq c$, $r<R$ and $\phi:K_1\rightarrow[0,+\infty)$ a continuous convex functional satisfying $\alpha(u_1)\le\phi(u_1)$, $u_1\in K_1$. Suppose that $T:\overline{\Omega}^\varphi_c\times(\overline{K}_2)_{r,R}\rightarrow \overline{\Omega}^\varphi_c\times K_2$ is a compact operator satisfying the following conditions
	\begin{enumerate}[(A)]
		\item
		\begin{enumerate}[(a)]
			\item $\{u\in \overline{\Omega}_c^\varphi[\alpha,\beta,a,b]\times (\overline{K}_2)_{r,R}: \alpha(u_1)>a\}\neq \emptyset$ and $\alpha(T_1u)>a$ for $u\in \overline{\Omega}_c^\varphi[\alpha,\beta,a,b]\times (\overline{K}_2)_{r,R}$ with $\alpha(u_1)=a$;
			\item $\phi(T_1 u)< d$ for $u\in\overline{\Omega}^\varphi_c[\phi,d]\times (\overline{K}_2)_{r,R}$;
			\item $\alpha(T_1 u)>a$ for $u\in K_1[\alpha,\varphi,a,c]\times (\overline{K}_2)_{r,R}$ with $\alpha(u_1)=a$ and $\beta(T_1u)>b$,
		\end{enumerate}
		\item For $u\in \overline{\Omega}^\varphi_c\times(\overline{K}_2)_{r,R}$
		\begin{enumerate}[(i)]
			\item 	$\|T_2 u\|>\|u_2\|$ if $\|u_2\|=r$ and $\|T_2 u\|<\|u_2\|$ if $\|u_2\|=R$; or
			\item	$\|T_2 u\|<\|u_2\|$ if $\|u_2\|=r$ and $\|T_2 u\|>\|u_2\|$ if $\|u_2\|=R$.
		\end{enumerate}
	\end{enumerate}
	Then, $T$ admits at least three distinct fixed points in $\overline{\Omega}^\varphi_c\times (K_2)_{r,R}$.
\end{theorem}

\noindent
{\bf Proof.} Assume that condition $(B)$--$(i)$ holds. The proof in the remaining case is analogous.
Under this assumption, the operator $T$ admits a fixed point in each of the following relatively open and pairwise disjoint subsets $\mathcal{U}_1:= K_1(\alpha,\varphi,a,c)\times (K_2)_{r,R},\,\mathcal{U}_2:=\overline{\Omega}^\varphi_c(\phi,d)\times (K_2)_{r,R},$ and 
\[\mathcal{U}_3:=\left(\overline{\Omega}_c^\varphi\times (K_2)_{r,R}\right)\setminus\left(\overline{\mathcal{U}}_1\cup\overline{\mathcal{U}}_2\right).\]

First observe that, since $T_1\bigl(\overline{\Omega}^\varphi_c \times (\overline{K}_2)_{r,R}\bigr)
\subset \overline{\Omega}^\varphi_c,$ it follows that
\[
\varphi(T_1 u) \le c
\quad
\text{for all }
u \in K_1[\alpha,\varphi,a,c] \times (\overline{K}_2)_{r,R}.
\]
Hence, by Proposition~\ref{Prop_Kras_LW}, $i_{\overline{\Omega}^\varphi_c \times K_2}\bigl(T,\mathcal{U}_1\bigr)=1,$ and therefore $T$ admits a fixed point in $\mathcal{U}_1$.

Consider next the extension of $T$ given by
\[
\tilde{T}:
\overline{\Omega}^{\varphi}_c\times\overline{\Omega}^{\|\cdot\|_2}_R
\longrightarrow K,
\qquad
\tilde{T}(u_1,u_2)
:=
T\bigl(u_1, \rho_2(u_2)\bigr),
\]
where $\rho_2$ again denotes a retraction, as in the previous proof.

Note that $\tilde{T}_1\left(\overline{\Omega}^\varphi_c\times \overline{\Omega}^{\|\cdot\|_2}_R\right)
\subset
\overline{\Omega}^\varphi_c.$ Moreover, for $u\in \overline{\Omega}^\varphi_c[\phi,d]\times \overline{\Omega}^{\|\cdot\|_2}_R$ it follows from $(A)$--$(b)$ that $\phi(\tilde{T}_1u)<d$, and hence $\tilde{T}_1\left(\overline{\Omega}^\varphi_c[\phi,d]\times \overline{\Omega}^{\|\cdot\|_2}_R\right) \subset\overline{\Omega}^\varphi_c(\phi,d).$ In addition, for every
$u\in \overline{\Omega}^\varphi_c\times\overline{\Omega}^{\|\cdot\|_2}_R$,
\[
\|\tilde{T}_2u\|>\|u_2\|
\, \text{if } \|u_2\|=r,
\text{ and }
\|\tilde{T}_2u\|<\|u_2\|
\, \text{if } \|u_2\|=R.
\]

Arguing as in the previous proof, we obtain
\[
\inf_{u\in\overline{\Omega}^\varphi_c[\phi,d]\times \overline{\Omega}^{\|\cdot\|_2}_r}
\|\tilde{T}_2u\|
=
\inf_{u\in \overline{\Omega}^\varphi_c[\phi,d]\times \partial_{K_2}\overline{\Omega}^{\|\cdot\|_2}_r}\|T_2u\|\geq \inf_{u\in\overline{\Omega}_c^\varphi\times \partial_{K_2}\overline{\Omega}^{\|\cdot\|_2}_r}
\|T_2u\|
>r>0,
\]
together with
\[\begin{aligned}
	\tilde{T}_2u\neq \lambda u_2
	\quad &
	\text{for }
	u_1\in \overline{\Omega}^\varphi_c,\ \|u_2\|=R,\ \lambda\ge1; \text{ and }\\
	u_2-\tilde{T}_2u\neq \mu \tilde{T}_2u
	\quad &
	\text{for }
	u_1\in \overline{\Omega}^\varphi_c,\ \|u_2\|=r,\ \mu\ge0.
\end{aligned}\]

We can therefore apply Proposition~\ref{prop3} both $(i)$ and $(ii)$, to $\tilde{T}$ and its restriction to $\overline{\Omega}^{\varphi}_c[\phi,d]\times \overline{\Omega}^{\|\cdot\|_2}_r$.
Consequently,
\[
\begin{aligned}
&i_{\overline{\Omega}^\varphi_c\times K_2}
\bigl(\tilde{T},\overline{\Omega}^\varphi_c\times \Omega^{\|\cdot\|_2}_R\bigr)
=
i_{\overline{\Omega}^\varphi_c\times K_2}
\bigl(\tilde{T},\overline{\Omega}^\varphi_c(\phi,d)\times \Omega^{\|\cdot\|_2}_R\bigr)
=1; \, \text{ and }\\
& i_{\overline{\Omega}^\varphi_c\times K_2}
\bigl(\tilde{T},\overline{\Omega}^\varphi_c\times \Omega^{\|\cdot\|_2}_r\bigr)
=
i_{\overline{\Omega}^\varphi_c\times K_2}
\bigl(\tilde{T},\overline{\Omega}^\varphi_c(\phi,d)\times \Omega^{\|\cdot\|_2}_r\bigr)
=0.
\end{aligned}
\]

By the additivity property of the fixed point index, we arrive to
\[\begin{aligned}
	i_{\overline{\Omega}^\varphi_c\times K_2}\left(\tilde{T},\,\overline{\Omega}^\varphi_c\times (K_2)_{r,R}\right)&=i_{\overline{\Omega}^\varphi_c\times K_2}\left(\tilde{T},\overline{\Omega}^\varphi_c\times \Omega^{\|\cdot\|_2}_R\right)- i_{\overline{\Omega}^\varphi_c\times K_2}\left(\tilde{T},\overline{\Omega}^\varphi_c\times \Omega^{\|\cdot\|_2}_r\right)=1,\\
	 i_{\overline{\Omega}^\varphi_c\times K_2}\left(\tilde{T},\,\mathcal{U}_2\right)&=i_{\overline{\Omega}^\varphi_c\times K_2}\left(\tilde{T},\overline{\Omega}^\varphi_c(\phi,d)\times \Omega^{\|\cdot\|_2}_R\right)- i_{\overline{\Omega}^\varphi_c\times K_2}\left(\tilde{T},\overline{\Omega}^\varphi_c(\phi,d)\times \Omega^{\|\cdot\|_2}_r\right)=1.
\end{aligned}\]
Applying additivity once more yields
\[i_{\overline{\Omega}^\varphi_c\times K_2}\left(\tilde{T},\,\mathcal{U}_3\right)=i_{\overline{\Omega}^\varphi_c\times K_2}\left(\tilde{T},\,\overline{\Omega}^\varphi_c\times (K_2)_{r,R}\right)-i_{\overline{\Omega}^\varphi_c\times K_2}\left(\tilde{T},\,\mathcal{U}_1\right)-i_{\overline{\Omega}^\varphi_c\times K_2}\left(\tilde{T},\,\mathcal{U}_2\right)=-1.\]

Finally, since $\tilde{T}=T$ on
$\overline{\Omega}^\varphi_c\times (\overline{K}_2)_{r,R}$,
we obtain $i_{\overline{\Omega}^\varphi_c \times K_2}\bigl(T, \mathcal{U}_2\bigr)=1,$ and $i_{\overline{\Omega}^\varphi_c \times K_2}\bigl(T, \mathcal{U}_3\bigr)=-1$.
The existence property of the fixed point index provides the existence of the remaining two fixed points of $T$ in $\mathcal{U}_2$ and $\mathcal{U}_3$ ( see Figure~\ref{fig1}).
\qed

\begin{figure}[htbp]
	\centering
	\subfigure[$i_{\overline{\Omega}^\varphi_c\times K_2}(T,\mathcal{U}_1)=1$.]{
		\includegraphics[scale=0.6]{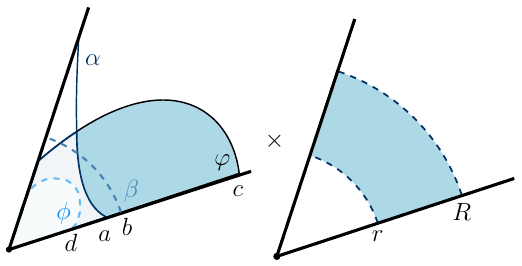}}
	\subfigure[$i_{\overline{\Omega}^\varphi_c\times K_2}(T,\mathcal{U}_2)=1$.]{\includegraphics[scale=0.6]{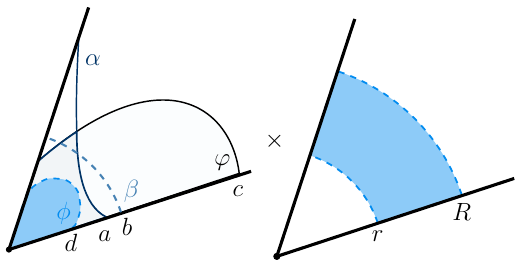}
	}
	\subfigure[$i_{\overline{\Omega}^\varphi_c\times K_2}(T,\mathcal{U}_3)=-1$.]{\includegraphics[scale=0.6]{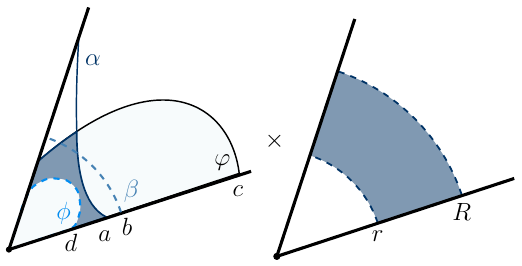}
	}
	\caption{Illustration of the three different sets for which a fixed point exists under the assumptions of Theorem~\ref{th_Kras_LW}.}
	\label{fig1}
\end{figure}

The requirement in Theorem~\ref{th_Kras_LW} that the image of $T_1$ lies in $\overline{\Omega}_c^\varphi$ can be removed, while still guaranteeing the existence of two solutions, one of which has both components nontrivial, as shown in the following result.

\begin{corollary}
		\label{Coro_Kras_LW}
	Let $a,c,d,r,R\in \mathbb{R}_+$ be such that $d<a<c$, $r<R$ and $\phi:K_1\rightarrow[0,+\infty)$ a continuous convex functional satisfying $\alpha(u_1)\le\phi(u_1)$, $u_1\in K_1$. Suppose that $T:\overline{\Omega}^\varphi_c\times(\overline{K}_2)_{r,R}\rightarrow K_1\times K_2$ is a compact operator satisfying the following conditions
	\begin{enumerate}[(A)]
		\item
		\begin{enumerate}[(a)]
			\item $\{u\in K_1[\alpha,\beta,a,c]\times (\overline{K}_2)_{r,R}: \alpha(u_1)>a\}\neq \emptyset$ and $\alpha(T_1u)>a$ for $u\in K_1[\alpha,\beta,a,c]\times (\overline{K}_2)_{r,R}$ with $\alpha(u_1)=a$;
			\item $\phi(T_1 u)< d$ for $u\in\overline{\Omega}^\varphi_c[\phi,d]\times (\overline{K}_2)_{r,R}$;
			\item $\alpha(T_1 u)>\frac{a}{c}\varphi(T_1u)$ for $u\in K_1[\alpha,\varphi,a,c]\times (\overline{K}_2)_{r,R}$ with $\alpha(u_1)=a$ and $\varphi(T_1u)>c$,
		\end{enumerate}
		\item For $u\in \overline{\Omega}^\varphi_c\times(\overline{K}_2)_{r,R}$
		\begin{enumerate}[(i)]
			\item 	$\|T_2 u\|>\|u_2\|$ if $\|u_2\|=r$ and $\|T_2 u\|<\|u_2\|$ if $\|u_2\|=R$; or
			\item	$\|T_2 u\|<\|u_2\|$ if $\|u_2\|=r$ and $\|T_2 u\|>\|u_2\|$ if $\|u_2\|=R$.
		\end{enumerate}
	\end{enumerate}
	Then, $T$ admits at least two distinct fixed points in $\overline{\Omega}^\varphi_c\times (K_2)_{r,R}$.
\end{corollary}

\noindent
{\bf Proof.} Let us define the auxiliary operator $N:\overline{\Omega}_c^\varphi\times(\overline{K}_2)_{r,R}
\longrightarrow \overline{\Omega}_c^\varphi\times K_2$ by
\begin{equation}
	N u=
	\begin{cases}
		T u, & \text{if } \varphi(T_1u)\leq c,\\[4pt]
		\left(\dfrac{c\,T_1 u}{\varphi(T_1 u)},\, T_2 u\right), & \text{if } \varphi(T_1u)> c.
	\end{cases}
\end{equation}

We claim that the operator $N$ satisfies all the assumptions of
Theorem~\ref{th_Kras_LW} with $\beta=\varphi$ and $b=c$.
Indeed, since $N_1\bigl(\overline{\Omega}_c^\varphi\times(\overline{K}_2)_{r,R}\bigr)
\subset \overline{\Omega}_c^\varphi$ there exists no element
$u\in K_1[\alpha,\varphi,a,c]\times(\overline{K}_2)_{r,R}$ such that
$\varphi(N_1u)>c$. Hence, condition \emph{(A)--(c)} is trivially satisfied.

Let $u_1\in \overline{\Omega}_c^\varphi[\phi,d]$ and
$u_2\in(\overline{K}_2)_{r,R}$. If $\varphi(T_1u)\leq c$, then
$N_1u=T_1u$ and, by condition \emph{(A)--(b)}, we obtain $\phi(N_1u)=\phi(T_1u)<d.$ On the other hand, if $\varphi(T_1u)>c$, then
$N_1u=\frac{c\,T_1u}{\varphi(T_1u)}$. By the convexity of $\phi$, it follows that
\[
\phi(N_1u)
=\phi\!\left(\frac{c\,T_1u}{\varphi(T_1u)}\right)
\leq \frac{c}{\varphi(T_1u)}\,\phi(T_1u)
<\phi(T_1u)<d.
\]
Therefore, operator $N$ satisfies condition \emph{(A)--(b)} of
Theorem~\ref{th_Kras_LW}.

Let now $u_1\in K_1[\alpha,\varphi,a,c]$ be such that $\alpha(u_1)=a$ and
$u_2\in(\overline{K}_2)_{r,R}$.
If $\varphi(T_1u)\leq c$, then $N_1u=T_1u$ and, by condition \emph{(A)--(a)}, $\alpha(N_1u)=\alpha(T_1u)>a$. If instead $\varphi(T_1u)>c$, we have
$N_1u=\dfrac{c\,T_1u}{\varphi(T_1u)}$ and, using the concavity of $\alpha$
together with condition \emph{(A)--(c)}, we infer that
\[
\alpha(N_1u)
=\alpha\!\left(\frac{c\,T_1u}{\varphi(T_1u)}\right)
\geq \frac{c}{\varphi(T_1u)}\,\alpha(T_1u)>a.
\]
Hence, condition \emph{(A)--(a)} of Theorem~\ref{th_Kras_LW}, with
$\beta=\varphi$ and $b=c$, is also fulfilled by operator $N$.

Condition \emph{(B)} of Theorem~\ref{th_Kras_LW} is clearly satisfied,
since $N_2u=T_2u$ for all
$u\in\overline{\Omega}_c^\varphi\times(\overline{K}_2)_{r,R}$.
Consequently, operator $N$ admits at least three fixed points in
$\overline{\Omega}_c^\varphi\times(K_2)_{r,R}$.

Moreover, the first component of two of these fixed points satisfies
$\alpha(u_1)<a$, as they are located in the sets $\overline{\Omega}_c^\varphi(\phi,d)\times(K_2)_{r,R}$ and $\bigl[\overline{\Omega}_c^\varphi(\alpha,a)\setminus
\overline{\Omega}_c^\varphi[\phi,d]\bigr]\times(K_2)_{r,R}$.

Let $\bar{u}$ be one of these two fixed points and assume, by contradiction,
that $\varphi(T_1\bar{u})>c$. Then
\[
a>\alpha(\bar{u}_1)
=\alpha(N_1\bar{u})
=\alpha\!\left(\frac{c\,T_1\bar{u}}{\varphi(T_1\bar{u})}\right)
\geq \frac{c}{\varphi(T_1\bar{u})}\,\alpha(T_1\bar{u})
>a,
\]
which is impossible. Therefore, necessarily
$\varphi(T_1\bar{u})\leq c$, and hence $\bar{u}=N\bar{u}=T\bar{u},$ showing that $\bar{u}$ is a fixed point of the original operator $T$. \qed

\begin{remark}
The conditions imposed on $T_1$ and $T_2$ in Proposition~\ref{Prop_Kras_LW}, Theorem~\ref{th_Kras_LW} and Corollary~\ref{Coro_Kras_LW} can evidently be interchanged without affecting the conclusion.
\end{remark}

\begin{remark}
	Note that in Theorem \ref{th_Kras_LW}  and Corollary~\ref{Coro_Kras_LW} it is possible to work with a more general domain for the second component, as discussed in \cite{Starshaped}. Specifically, one may consider a compact operator
	\[
	T:\overline{\Omega}^\varphi_c\times(\overline{\mathcal{O}}_2\setminus\Omega_2)\longrightarrow K_1\times K_2,
	\]
 where $\Omega_2$ and $\mathcal{O}_2$ are relatively open bounded subsets of the cone $K_2$, satisfying $0\in \Omega_2\subset\overline{\Omega}_2\subset \mathcal{O}_2$, and $\overline{\Omega}_2$ is a strictly star-shaped set in $K_2$. In this setting, condition $(B)$ must be adapted as follows:
	\begin{enumerate}[(B)]
		\item For $u\in \overline{\Omega}^\varphi_c\times(\overline{\mathcal{O}}_2\setminus\Omega_2)$:
		\begin{enumerate}[(i)]
			\item $\|T_2 u\|>\|u_2\|$ if $u_2\in\partial_{K_2}\Omega_2$ and $\|T_2 u\|<\|u_2\|$ if $u_2\in\partial_{K_2}\mathcal{O}_2$; or
			\item $\|T_2 u\|<\|u_2\|$ if $u_2\in\partial_{K_2}\Omega_2$ and $\|T_2 u\|>\|u_2\|$ if $u_2\in\partial_{K_2}\mathcal{O}_2$.
		\end{enumerate}
	\end{enumerate}
	
Moreover, although norm-type conditions are imposed, one may also consider other-types of compression-expansion conditions. Further details are provided in \cite[Remark 3.3]{Starshaped} (see also \cite[Corollary 3.1]{Starshaped}).
\end{remark}

\section{Leggett-Williams fixed point  theorem in product spaces}

In this section, we present a new multiplicity result under Leggett-Williams type conditions on both components. Accordingly, we introduce the following notation.

 For $j\in \{1,2\}$, let $\alpha_j,\beta_j:K_j\rightarrow [0,\infty)$ be continuous functionals on $K_j$, with $\alpha_j$ concave and $\beta_j$ convex. In addition, let $\phi_j,\varphi_j:K_j\rightarrow[0,\infty)$ be  continuous convex functionals  and suppose $\alpha_j(u)\le\phi_j(u)$, $u\in K_j$.
 
For $a_j,b_j, c_j \in \mathbb{R}_+$ satisfying $a_j <b_j\le c_j$, let us denote
 \[\begin{aligned}
 	K[\alpha,\varphi,a,c]&=K_1[\alpha_1,\varphi_1,a_1,c_1]
 	\times
 	K_2[\alpha_2,\varphi_2,a_2,c_2],\\
 	K(\alpha,\varphi,a,c) &=K_1(\alpha_1,\varphi_1,a_1,c_1)
 	\times
 	K_2(\alpha_2,\varphi_2,a_2,c_2),\\
 	\overline{\Omega}_c^\varphi[\alpha,\beta,a,b]&=\overline{\Omega}_{c_1}^{\varphi_1}[\alpha_1,\beta_1,a_1,b_1]\times \overline{\Omega}_{c_2}^{\varphi_2}[\alpha_2,\beta_2,a_2,b_2],\\
 	\overline{\Omega}_c^\varphi(\alpha,\beta,a,b)&=\overline{\Omega}_{c_1}^{\varphi_1}(\alpha_1,\beta_1,a_1,b_1)\times \overline{\Omega}_{c_2}^{\varphi_2}(\alpha_2,\beta_2,a_2,b_2).
 \end{aligned}\]
 
 Following a strategy similar to that of the previous section, we first compute the fixed point index of the operator on a set of the form $K(\alpha,\varphi,a,c)$, so that the existence of a fixed point is guaranteed.

\begin{theorem}
	\label{th1}
	Let $a_j,b_j,c_j\in \mathbb{R}_+$ be such that $a_j<b_j\leq c_j$ $(j=1,2)$. Suppose that $T:K[\alpha,\varphi,a,c]\rightarrow K$ is a compact operator such that for both $j\in\{1,2\}$ the following conditions are satisfied:
	\begin{enumerate}[(a)]
	\item $\{u_j\in \overline{\Omega}_{c_j}^{\varphi_j}[\alpha_j,\beta_j,a_j,b_j]: \alpha_j(u_j)>a_j\, \}\neq \emptyset$ and $\alpha_j(T_ju)>a_j$ for $u\in \overline{\Omega}_c^\varphi[\alpha,\beta,a,b]$ with $\alpha_j(u_j)=a_j$;
	\item $\varphi_j(T_j u)\leq c_j$ for $u\in K[\alpha,\varphi,a,c]$;
	\item $\alpha_j(T_j u)>a_j$ for $u\in K[\alpha,\varphi,a,c]$ with $\alpha_j(u_j)=a_j$ and $\beta_j(T_ju)>b_j$.
	\end{enumerate}
	Then $i_{\overline{\Omega}^{\varphi_1}_{c_1}\times \overline{\Omega}^{\varphi_2}_{c_2}}\left(T, K(\alpha,\varphi, a, c)\right)=1$. Consequently, $T$ admits at least one fixed point in $K(\alpha,\varphi, a, c)$.
\end{theorem}

\noindent
{\bf Proof.} Observe that $T$ has no fixed points on
\[\begin{aligned}
\partial_{
	\overline{\Omega}^{\varphi_1}_{c_1}
	\times
	\overline{\Omega}^{\varphi_2}_{c_2}}
K(\alpha,\varphi,a,c),
\end{aligned}
\]
that is, there is no $u=(u_1,u_2) \in K[\alpha,\varphi,a,c]$ such that $Tu=u$ and either $\alpha_1(u_1)=a_1$ or $\alpha_2(u_2)=a_2$. The reasoning parallels that of Proposition~\ref{prop1} and is omitted. 

For each $j \in \{1,2\}$, choose $\bar u_j \in K_j[\alpha_j,\varphi_j,a_j,c_j]$ satisfying $\alpha_j(\bar u_j) > a_j$, and define the homotopy
\[
h : K[\alpha,\varphi,a,c] \times [0,1] \longrightarrow \overline{\Omega}^{\varphi_1}_{c_1}\times \overline{\Omega}^{\varphi_2}_{c_2}, \qquad
h(u_1,u_2,t) := \bigl( t T_1 u + (1-t) \bar u_1,\; t T_2 u + (1-t) \bar u_2 \bigr).
\]
This defines a compact homotopy. Additionally, by the same reasoning as in Proposition~\ref{prop1}, one verifies that $h(u,t) \neq u$ for all $u\in \partial_{
	\overline{\Omega}^{\varphi_1}_{c_1}
	\times
	\overline{\Omega}^{\varphi_2}_{c_2}}
K(\alpha,\varphi,a,c)$ and $t\in[0,1]$. It follows then from the homotopy invariance and normalization properties of the fixed point index that
\[
i_{\overline{\Omega}^{\varphi_1}_{c_1} \times \overline{\Omega}^{\varphi_2}_{c_2}} \bigl(T, K(\alpha,\varphi,a,c)\bigr)
= i_{\overline{\Omega}^{\varphi_1}_{c_1} \times \overline{\Omega}^{\varphi_2}_{c_2}} \bigl((\bar u_1, \bar u_2), K(\alpha,\varphi,a,c)\bigr)
= 1,
\]
which completes the proof. \qed

To prove the main theorem of this section, we first introduce the following result, which can be regarded as a consequence of Proposition \ref{prop1}.

\begin{lemma}
	\label{lemma}
	Let $a_1,b_1,c_1\in \mathbb{R}_+$ be such that $a_1<b_1\leq c_1$. Suppose that $T:\overline{\Omega}^{\varphi_1}_{c_1}\times \overline{\Omega}^{\varphi_2}_{c_2}\rightarrow K_1\times \overline{\Omega}^{\varphi_2}_{c_2}$ is a compact operator such that 
		\begin{enumerate}[(a)]
		\item $\{u_1\in \overline{\Omega}_{c_1}^{\varphi_1}[\alpha_1,\beta_1,a_1,b_1]: \alpha_1(u_1)>a_1\, \}\neq \emptyset$ and $\alpha_1(T_1u)>a_1$ for $u\in \overline{\Omega}_{c_1}^{\varphi_1}[\alpha_1,\beta_1,a_1,b_1]\times \overline{\Omega}_{c_2}^{\varphi_2}$ with $\alpha_1(u_1)=a_1$;
		\item $\varphi_1(T_1 u)\leq c_1$ for $u_1\in K_1[\alpha_1,\varphi_1,a_1,c_1]\times \overline{\Omega}^{\varphi_2}_{c_2}$;
		\item $\alpha_1(T_1 u)>a_1$ for $u_1\in K_1[\alpha_1,\varphi_1,a_1,c_1]\times \overline{\Omega}^{\varphi_2}_{c_2}$ with $\alpha_1(u_1)=a_1$ and $\beta_1(T_1u)>b_1$.
	\end{enumerate}
	Then 
	\[i_{\overline{\Omega}^{\varphi_1}_{c_1}\times \overline{\Omega}^{\varphi_2}_{c_2}}\left(T, K_1(\alpha_1,\varphi_1, a_1, c_1)\times\overline{\Omega}^{\varphi_2}_{c_2}\right)=1.\]
\end{lemma}

\noindent
{\bf Proof.} Fix $u\in \overline{\Omega}^{\varphi_1}_{c_1}\times \overline{\Omega}^{\varphi_2}_{c_2}$. We claim that
\[
T_2 u \neq \lambda u_2
\quad \text{for } \varphi_2(u_2)=c_2
\text{ and all } \lambda > 1.
\]
Indeed, suppose by contradiction that there exist $\bar u\in\overline{\Omega}^{\varphi_1}_{c_1}\times \overline{\Omega}^{\varphi_2}_{c_2}$ with $\varphi_2(\bar{u}_2)=c_2$ and $\tilde\lambda> 1$ such that $T_2\bar u = \tilde\lambda \bar u_2$. Since $\varphi_2$ is convex, we obtain
\[
c_2 = \varphi_2(\bar u_2)
= \varphi_2\!\left(\frac{1}{\tilde\lambda}T_2\bar u\right)
\leq \frac{1}{\tilde\lambda}\varphi_2(T_2\bar u),
\]
which implies $\varphi_2(T_2\bar u)> c_2$, contradicting $T_2(\overline{\Omega}^{\varphi_1}_{c_1}\times \overline{\Omega}^{\varphi_2}_{c_2})\subset\overline{\Omega}^{\varphi_2}_{c_2}$. Thus, by applying the reasoning from the proof of Proposition \ref{prop1} and considering $\overline{\Omega}^{\varphi_1}_{c_1}\times \overline{\Omega}^{\varphi_2}_{c_2}$ as the retract for the computation of the index, the result follows. \qed

Using these computations, we now establish the vectorial version of the Leggett-Williams fixed point theorem, which ensures the existence of nine distinct fixed points.
\begin{theorem}
	\label{th2}
		\emph{(Vectorial version of Leggett-Williams fixed point theorem)}
	Let $a_j,b_j,c_j,d_j\in \mathbb{R}_+$ be such that $d_j<a_j<b_j\leq c_j$ $(j=1,2)$. Suppose that $T:\overline{\Omega}^{\varphi_1}_{c_1}\times \overline{\Omega}^{\varphi_2}_{c_2}\rightarrow \overline{\Omega}^{\varphi_1}_{c_1}\times \overline{\Omega}^{\varphi_2}_{c_2}$ is a compact operator such that for both $j\in\{1,2\}$ the following conditions are satisfied:
	\begin{enumerate}[(a)]
		\item $\{u_j\in \overline{\Omega}_{c_j}^{\varphi_j}[\alpha_j,\beta_j,a_j,b_j]: \alpha_j(u_j)>a_j\, \}\neq \emptyset$ and $\alpha_j(T_ju)>a_j$ for $u_j\in \overline{\Omega}_{c_j}^{\varphi_j}[\alpha_j,\beta_j,a_j,b_j]$, $ u_i\in \overline{\Omega}_{c_i}^{\varphi_i}$ with $\alpha_j(u_j)=a_j$ $(i\neq j)$;
		\item $\phi_j(T_j u)< d_j$ for $u_j\in\overline{\Omega}^{\varphi_j}_{c_j}[\phi_j,d_j]$, $u_i\in \overline{\Omega}^{\varphi_i}_{c_i}$ $(i\neq j)$;
		\item $\alpha_j(T_j u)>a_j$ for $u_j\in K_j[\alpha_j,\varphi_j,a_j,c_j]$, $u_i\in \overline{\Omega}^{\varphi_i}_{c_i}$ with $\alpha_j(u_j)=a_j$ and $\beta_j(T_ju)>b_j$ $(i\neq j)$.
	\end{enumerate}
	Then, $T$ admits at least nine distinct fixed points in $\overline{\Omega}^{\varphi_1}_{c_1}\times \overline{\Omega}^{\varphi_2}_{c_2}$.
\end{theorem}

\noindent
{\bf Proof.} Denote $\mathcal{U}_1 := K(\alpha,\varphi,a,c)$. Observe that $T_j\bigl(\overline{\Omega}^{\varphi_1}_{c_1} \times \overline{\Omega}^{\varphi_2}_{c_2}\bigr) \subset \overline{\Omega}^{\varphi_j}_{c_j}, \, j \in \{1,2\}.$ Hence,
\[
\varphi_j(T_j u) \le c_j \quad \text{for } u_j \in K_j[\alpha_j,\varphi_j,a_j,c_j],\ u_i \in \overline{\Omega}^{\varphi_i}_{c_i}\ (i \neq j).
\]
It follows from Theorem~\ref{th1} that $i_{\overline{\Omega}^{\varphi_1}_{c_1} \times \overline{\Omega}^{\varphi_2}_{c_2}}(T, \mathcal{U}_1) = 1,$ ensuring the existence of a fixed point of $T$ in $\mathcal{U}_1$.

For the sake of clarity in the sequel, we shall employ the following notation
\[
\begin{aligned}
	&\mathcal{U}_2 := K_1(\alpha_1,\varphi_1,a_1,c_1)\times \overline{\Omega}^{\varphi_2}_{c_2}(\phi_2,d_2), \qquad
	&&\mathcal{U}_3 :=\overline{\Omega}^{\varphi_1}_{c_1}(\phi_1,d_1)\times K_2(\alpha_2,\varphi_2,a_2,c_2),\\
	&\mathcal{V}_1 := K_1(\alpha_1,\varphi_1,a_1,c_1)\times \overline{\Omega}^{\varphi_2}_{c_2}, \qquad
	&&\mathcal{V}_2 := \overline{\Omega}^{\varphi_1}_{c_1}\times K_2(\alpha_2,\varphi_2,a_2,c_2).
\end{aligned}
\]

Fix $j\in\{1,2\}$ and let $u_j\in \overline{\Omega}^{\varphi_j}_{c_j}$. It is clear that
\[
T_i u \neq \lambda u_i
\quad \text{for } \phi_i(u_i)=d_i
\text{ and all } \lambda \geq 1,\quad (i\neq j).
\]
Indeed, suppose by contradiction that there exist $\bar u\in \overline{\Omega}^{\varphi_1}_{c_1}\times \overline{\Omega}^{\varphi_2}_{c_2}$ with $\phi_i(\bar u_i)=d_i$ and $\tilde\lambda\geq 1$ such that $T_i\bar u = \tilde\lambda \bar u_i$. Then, since $\phi_i$ is convex, we obtain
\[
d_i = \phi_i(\bar u_i)
= \phi_i\!\left(\frac{1}{\tilde\lambda}T_i\bar u\right)
\leq \frac{1}{\tilde\lambda}\phi_i(T_i\bar u),
\]
which implies $\phi_i(T_i\bar u)\geq d_i$, contradicting assumption~$(b)$.  Hence, Proposition~\ref{prop1} applies to $T$, yielding
\[
	i_{\overline{\Omega}^{\varphi_1}_{c_1}\times \overline{\Omega}^{\varphi_2}_{c_2}}(T,\mathcal{U}_2)
	=i_{\overline{\Omega}^{\varphi_1}_{c_1}\times \overline{\Omega}^{\varphi_2}_{c_2}}(T,\mathcal{U}_3)
	=1.
\]

On the other hand, for a $j\in\{1,2\}$, conditions $(a)$--$(b)$--$(c)$ are satisfied together with $T_i(\overline{\Omega}^{\varphi_1}_{c_1}\times \overline{\Omega}^{\varphi_2}_{c_2})\subset\overline{\Omega}^{\varphi_i}_{c_i}$ $(i\neq j)$. Therefore, by applying Lemma \ref{lemma}, we find
\[i_{\overline{\Omega}^{\varphi_1}_{c_1}\times \overline{\Omega}^{\varphi_2}_{c_2}}(T,\mathcal{V}_1)
=i_{\overline{\Omega}^{\varphi_1}_{c_1}\times \overline{\Omega}^{\varphi_2}_{c_2}}(T,\mathcal{V}_2)=1.\]

Let us introduce the following notation for what follows
\[
	\mathcal{U}_4 :=\overline{\Omega}^{\varphi_1}_{c_1}(\phi_1,d_1)\times \overline{\Omega}^{\varphi_2}_{c_2}(\phi_2,d_2), \,
	\mathcal{V}_3 := \overline{\Omega}^{\varphi_1}_{c_1}(\phi_1,d_1)\times \overline{\Omega}^{\varphi_2}_{c_2},\,
	\mathcal{V}_4 := \overline{\Omega}^{\varphi_1}_{c_1}\times \overline{\Omega}^{\varphi_2}_{c_2}(\phi_2,d_2), \,
	\mathcal{V}_5 := \overline{\Omega}^{\varphi_1}_{c_1}\times \overline{\Omega}^{\varphi_2}_{c_2}.
\]

Note that $\overline{\mathcal{U}}_4$, $\overline{\mathcal{V}}_3$, $\overline{\mathcal{V}}_4$, and $\overline{\mathcal{V}}_5$ are all retracts of $\overline{\Omega}^{\varphi_1}_{c_1} \times \overline{\Omega}^{\varphi_2}_{c_2}$. It then follows from $(b)$ that Proposition~\ref{prop_index} can be applied to $T$, obtaining
\[
	i_{\overline{\Omega}^{\varphi_1}_{c_1}\times \overline{\Omega}^{\varphi_2}_{c_2}}(T,\mathcal{U}_4)
	=i_{\overline{\Omega}^{\varphi_1}_{c_1}\times \overline{\Omega}^{\varphi_2}_{c_2}}(T,\mathcal{V}_3)
	=i_{\overline{\Omega}^{\varphi_1}_{c_1}\times \overline{\Omega}^{\varphi_2}_{c_2}}(T,\mathcal{V}_4)
	=i_{\overline{\Omega}^{\varphi_1}_{c_1}\times \overline{\Omega}^{\varphi_2}_{c_2}}(T,\mathcal{V}_5)
	=1.
\]

We now compute the fixed point index of $T$ on the sets (see Figure \ref{fig2})
\[
\begin{aligned}
	&\mathcal{U}_5 := \mathcal{V}_1\setminus(\overline{\mathcal{U}}_1\cup \overline{\mathcal{U}}_2), \qquad
	\mathcal{U}_6 := \mathcal{V}_2\setminus(\overline{\mathcal{U}}_1\cup \overline{\mathcal{U}}_3), \qquad
	\mathcal{U}_7 := \mathcal{V}_3\setminus(\overline{\mathcal{U}}_3\cup \overline{\mathcal{U}}_4), \\
	&\mathcal{U}_8 := \mathcal{V}_4\setminus(\overline{\mathcal{U}}_2\cup \overline{\mathcal{U}}_4), \qquad
	\mathcal{V}_6 := \mathcal{V}_5\setminus(\overline{\mathcal{V}}_1\cup \overline{\mathcal{V}}_3), \qquad
	\mathcal{U}_9 := \mathcal{V}_6\setminus(\overline{\mathcal{U}}_6\cup \overline{\mathcal{U}}_8),
\end{aligned}
\]
by means of the additivity property of the fixed point index
\[\begin{aligned}
	&i_{\overline{\Omega}^{\varphi_1}_{c_1}\times\overline{\Omega}^{\varphi_2}_{c_2}}(T,\mathcal{U}_5)=i_{\overline{\Omega}^{\varphi_1}_{c_1}\times\overline{\Omega}^{\varphi_2}_{c_2}}(T,\mathcal{V}_1)-i_{\overline{\Omega}^{\varphi_1}_{c_1}\times\overline{\Omega}^{\varphi_2}_{c_2}}(T,\mathcal{U}_1)-i_{\overline{\Omega}^{\varphi_1}_{c_1}\times\overline{\Omega}^{\varphi_2}_{c_2}}(T,\mathcal{U}_2)=-1,\\
	&i_{\overline{\Omega}^{\varphi_1}_{c_1}\times\overline{\Omega}^{\varphi_2}_{c_2}}(T,\mathcal{U}_6)=i_{\overline{\Omega}^{\varphi_1}_{c_1}\times\overline{\Omega}^{\varphi_2}_{c_2}}(T,\mathcal{V}_2)-i_{\overline{\Omega}^{\varphi_1}_{c_1}\times\overline{\Omega}^{\varphi_2}_{c_2}}(T,\mathcal{U}_1)-i_{\overline{\Omega}^{\varphi_1}_{c_1}\times\overline{\Omega}^{\varphi_2}_{c_2}}(T,\mathcal{U}_3)=-1,\\
	&i_{\overline{\Omega}^{\varphi_1}_{c_1}\times\overline{\Omega}^{\varphi_2}_{c_2}}(T,\mathcal{U}_7)=i_{\overline{\Omega}^{\varphi_1}_{c_1}\times\overline{\Omega}^{\varphi_2}_{c_2}}(T,\mathcal{V}_3)-i_{\overline{\Omega}^{\varphi_1}_{c_1}\times\overline{\Omega}^{\varphi_2}_{c_2}}(T,\mathcal{U}_3)-i_{\overline{\Omega}^{\varphi_1}_{c_1}\times\overline{\Omega}^{\varphi_2}_{c_2}}(T,\mathcal{U}_4)=-1,\\
	&i_{\overline{\Omega}^{\varphi_1}_{c_1}\times\overline{\Omega}^{\varphi_2}_{c_2}}(T,\mathcal{U}_8)=i_{\overline{\Omega}^{\varphi_1}_{c_1}\times\overline{\Omega}^{\varphi_2}_{c_2}}(T,\mathcal{V}_4)-i_{\overline{\Omega}^{\varphi_1}_{c_1}\times\overline{\Omega}^{\varphi_2}_{c_2}}(T,\mathcal{U}_2)-i_{\overline{\Omega}^{\varphi_1}_{c_1}\times\overline{\Omega}^{\varphi_2}_{c_2}}(T,\mathcal{U}_4)=-1,\\
	&i_{\overline{\Omega}^{\varphi_1}_{c_1}\times\overline{\Omega}^{\varphi_2}_{c_2}}(T,\mathcal{V}_6)=i_{\overline{\Omega}^{\varphi_1}_{c_1}\times\overline{\Omega}^{\varphi_2}_{c_2}}(T,\mathcal{V}_5)-i_{\overline{\Omega}^{\varphi_1}_{c_1}\times\overline{\Omega}^{\varphi_2}_{c_2}}(T,\mathcal{V}_1)-i_{\overline{\Omega}^{\varphi_1}_{c_1}\times\overline{\Omega}^{\varphi_2}_{c_2}}(T,\mathcal{V}_3)=-1,\\
	&i_{\overline{\Omega}^{\varphi_1}_{c_1}\times\overline{\Omega}^{\varphi_2}_{c_2}}(T,\mathcal{U}_9)=i_{\overline{\Omega}^{\varphi_1}_{c_1}\times\overline{\Omega}^{\varphi_2}_{c_2}}(T,\mathcal{V}_6)-i_{\overline{\Omega}^{\varphi_1}_{c_1}\times\overline{\Omega}^{\varphi_2}_{c_2}}(T,\mathcal{U}_6)-i_{\overline{\Omega}^{\varphi_1}_{c_1}\times\overline{\Omega}^{\varphi_2}_{c_2}}(T,\mathcal{U}_8)=1.
\end{aligned}\]

Finally, the existence property of the fixed point index guarantees the existence of at least one fixed point of $T$ in each of the sets $\mathcal{U}_k$, $k\in\{1,\dots,9\}$. Moreover, all these sets are pairwise distinct. \qed

\begin{figure}[h]
	\centering
	\subfigure[$i_{\overline{\Omega}^{\varphi_1}_{c_1}\times\overline{\Omega}^{\varphi_2}_{c_2}}(T,\mathcal{U}_1)=1$.]{
		\includegraphics[scale=0.58]{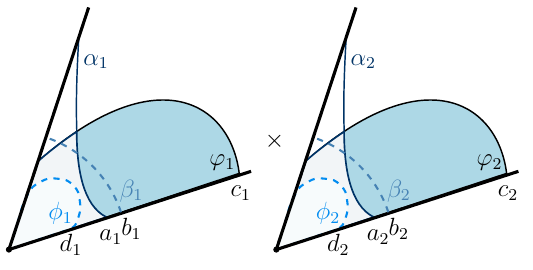}}
	\subfigure[$i_{\overline{\Omega}^{\varphi_1}_{c_1}\times\overline{\Omega}^{\varphi_2}_{c_2}}(T,\mathcal{U}_2)=1$.]{\includegraphics[scale=0.58]{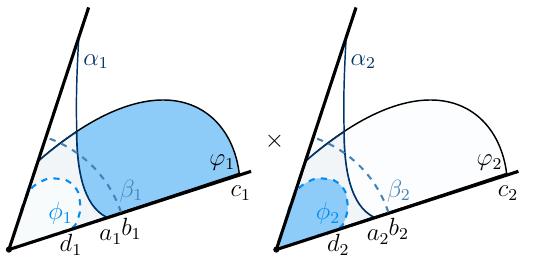}
	}
	\subfigure[$i_{\overline{\Omega}^{\varphi_1}_{c_1}\times\overline{\Omega}^{\varphi_2}_{c_2}}(T,\mathcal{U}_3)=1$.]{\includegraphics[scale=0.58]{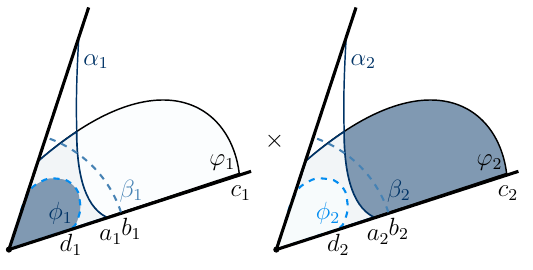}
	}
	\subfigure[$i_{\overline{\Omega}^{\varphi_1}_{c_1}\times\overline{\Omega}^{\varphi_2}_{c_2}}(T,\mathcal{U}_4)=1$.]{
		\includegraphics[scale=0.58]{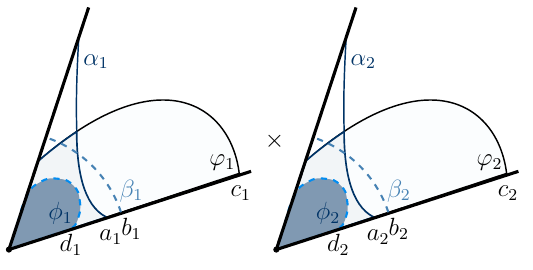}}
	\subfigure[$i_{\overline{\Omega}^{\varphi_1}_{c_1}\times\overline{\Omega}^{\varphi_2}_{c_2}}(T,\mathcal{U}_5)=-1$.]{\includegraphics[scale=0.58]{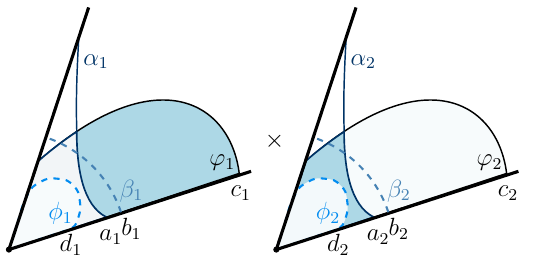}
	}
	\subfigure[$i_{\overline{\Omega}^{\varphi_1}_{c_1}\times\overline{\Omega}^{\varphi_2}_{c_2}}(T,\mathcal{U}_6)=-1$.]{\includegraphics[scale=0.58]{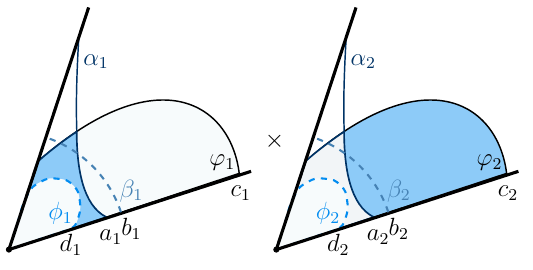}
	}
	\subfigure[$i_{\overline{\Omega}^{\varphi_1}_{c_1}\times\overline{\Omega}^{\varphi_2}_{c_2}}(T,\mathcal{U}_7)=-1$.]{
		\includegraphics[scale=0.58]{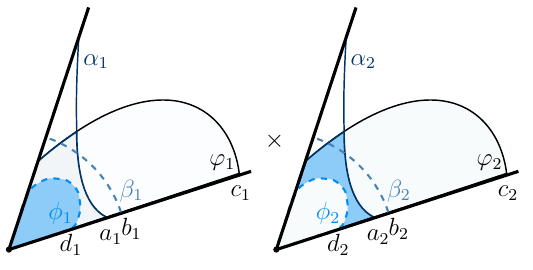}}
	\subfigure[$i_{\overline{\Omega}^{\varphi_1}_{c_1}\times\overline{\Omega}^{\varphi_2}_{c_2}}(T,\mathcal{U}_8)=-1$.]{\includegraphics[scale=0.58]{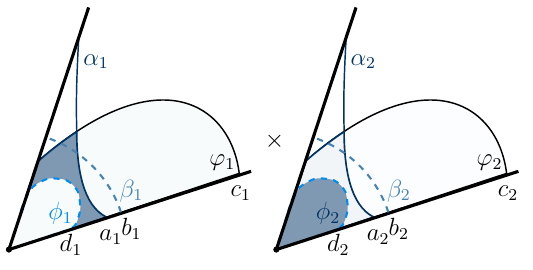}
	}
	\subfigure[$i_{\overline{\Omega}^{\varphi_1}_{c_1}\times\overline{\Omega}^{\varphi_2}_{c_2}}(T,\mathcal{U}_9)=1$.]{\includegraphics[scale=0.58]{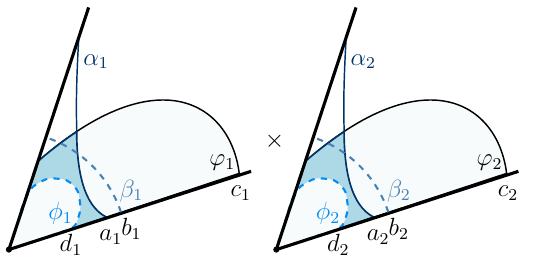}
	}
	\caption{Localization of the nine fixed points under the assumptions of Theorem~\ref{th2}.}
	\label{fig2}
\end{figure}

\begin{remark}
	Under the assumptions of Theorem~\ref{th2}, the operator possesses at least four coexistence fixed points, which are localized in \(\mathcal{U}_1\), \(\mathcal{U}_5\), \(\mathcal{U}_6\), and \(\mathcal{U}_9\), respectively. This is clearly illustrated by the sets depicted in Figure~\ref{fig2}.
\end{remark}

As in the previous section, in Theorem~\ref{th2} we may relax the requirement $T\bigl(\overline{\Omega}_{c_1}^{\varphi_1}\times\overline{\Omega}_{c_2}^{\varphi_2}\bigr)\subset\overline{\Omega}_{c_1}^{\varphi_1}\times\overline{\Omega}_{c_2}^{\varphi_2},$ while still ensuring four solutions, one with both components nontrivial, as established in the following corollary by the appropriate modification of condition \emph{(c)}.

\begin{corollary}
		\label{coro4.1}
Let $a_j,c_j,d_j\in \mathbb{R}_+$ be such that $d_j<a_j<c_j$ $(j=1,2)$. Suppose that $T:\overline{\Omega}^{\varphi_1}_{c_1}\times \overline{\Omega}^{\varphi_2}_{c_2}\rightarrow K_1\times K_2$ is a compact operator such that for both $j\in\{1,2\}$ the following conditions are satisfied:
\begin{enumerate}[(a)]
	\item $\{u_j\in K_j[\alpha_j,\beta_j,a_j,c_j]: \alpha_j(u_j)>a_j\, \}\neq \emptyset$ and $\alpha_j(T_ju)>a_j$ for $u_j\in K_j[\alpha_j,\beta_j,a_j,c_j]$, $ u_i\in \overline{\Omega}_{c_i}^{\varphi_i}$ with $\alpha_j(u_j)=a_j$, $(i\neq j)$;
	\item $\phi_j(T_j u)< d_j$ for $u_j\in\overline{\Omega}^{\varphi_j}_{c_j}[\phi_j,d_j]$, $u_i\in \overline{\Omega}^{\varphi_i}_{c_i}$ $(i\neq j)$;
	\item $\alpha_j(T_j u)>\frac{a_j}{c_j}\varphi_j(T_j u)$ for $u_j\in K_j[\alpha_j,\varphi_j,a_j,c_j]$, $u_i\in \overline{\Omega}^{\varphi_i}_{c_i}$ with $\alpha_j(u_j)=a_j$ and $\varphi_j(T_ju)>c_j$ $(i\neq j)$.
\end{enumerate}
Then, $T$ admits at least four distinct fixed points in $\overline{\Omega}^{\varphi_1}_{c_1}\times \overline{\Omega}^{\varphi_2}_{c_2}$.
\end{corollary}

\noindent
{\bf Proof.} Consider the auxiliary operator given by 
\begin{equation}
	N u=
	\begin{cases}
		T u, & \text{if } \varphi_j(T_ju)\leq c_j, \, j\in\{1,2\},\\[8pt]
		\left(\dfrac{c_1\,T_1 u}{\varphi_1(T_1 u)},\, T_2 u\right), & \text{if } \varphi_1(T_1u)> c_1, \,\varphi_2(T_2 u)\le c_2,\\[8pt]
		\left(T_1 u, \dfrac{c_2\,T_2 u}{\varphi_2(T_2 u)}\right),& \text{if } \varphi_1(T_1u)\le c_1, \,\varphi_2(T_2 u)> c_2,\\[8pt]
		\left(\dfrac{c_1\,T_1 u}{\varphi_1(T_1 u)}, \dfrac{c_2\,T_2 u}{\varphi_2(T_2 u)}\right), & \text{if } \varphi_j(T_ju)> c_j,\,  j\in\{1,2\}.
	\end{cases}
\end{equation}

It is straightforward to verify that each component of the operator $N$
fulfills the conditions of Theorem~\ref{th2} with
$\beta_j=\varphi_j$ and $b_j=c_j$ for $j=1,2$, by similar reasoning to that in
Corollary~\ref{Coro_Kras_LW}.
Consequently, operator $N$ admits nine fixed points, four of which satisfy
$\alpha_j(u_j)<a_j$ for both $j\in\{1,2\}$ (i.e., $u_j$ denotes the $j$-th
component of a fixed point).

Let $\bar{u}$ be one of these four fixed points and suppose, by contradiction,
that $\varphi_j(T_j\bar{u})>c_j$ for $j\in\{1,2\}$. The concavity of the functional $\alpha_j$ yields
\[
a_j>\alpha_j(\bar{u}_j)
=\alpha_j(N_j\bar{u})
=\alpha_j\!\left(\frac{c_j\,T_j\bar{u}}{\varphi_j(T_j\bar{u})}\right)
\geq \frac{c_j}{\varphi_j(T_j\bar{u})}\,\alpha_j(T_j\bar{u})
>a_j,
\]
which is impossible. Therefore, necessarily
$\varphi_j(T_j\bar{u})\leq c_j$ for $j\in\{1,2\}$, and hence $\bar{u}=N\bar{u}=T\bar{u},$ showing that $\bar{u}$ is a fixed point of the original operator $T$. \qed

\section{Applications to systems of nonlinear differential equations}

First, we consider the following system of second-order differential equations
\begin{equation}
	\label{5.3}
	\left\{
	\begin{aligned}
		& u_1''(t) + f_1(u_1(t), u_2(t)) = 0, \quad t\in[0,1], \\
		& u_2''(t) + f_2(u_1(t), u_2(t)) = 0, \quad t\in[0,1],\\
		& u_1(0) = u_1'(1) = 0= u_2(0) = u_2'(1), 
	\end{aligned}
	\right.
\end{equation}
where $f_1, f_2:[0,\infty)^2   \to [0,\infty)$ are continuous functions.

Let \(\mathcal{X} := \mathcal{C}([0,1]; \mathbb{R})\) be the Banach space of continuous functions on \([0,1]\) endowed with the supremum norm $\|\cdot\|_\infty$, and let \(\mathcal{X}_+ := \mathcal{C}_+([0,1]; \mathbb{R})\) denote the cone of nonnegative continuous functions. It is well known that solutions of the system \eqref{5.3} are precisely the fixed points of the operator $T = (T_1, T_2): \mathcal{X}_+  \times \mathcal{X}_+  \to \mathcal{X} \times \mathcal{X}$ given by
\begin{equation}
	\label{5.4}
	\begin{aligned}
		T_{1}(u_{1}, u_{2})(t) &= \int_{0}^{1} G_{1}(t, s) f_{1}(u_{1}(s), u_{2}(s)) \, ds, \quad t \in [0, 1]; \\
		T_{2}(u_{1}, u_{2})(t) &= \int_{0}^{1} G_{2}(t, s) f_{2}(u_{1}(s), u_{2}(s)) \, ds, \quad t \in [0, 1],
	\end{aligned}
\end{equation}
where the Green's functions are of the form
\[G_1(t,s)=G_2(t,s)=\min\{t,s\}, \text{ for }t\in[0,1] \text{ and } s\in[0,1].\]

\begin{remark}
	\label{re5.1}
	Under the continuity assumption on the nonlinearities, the operator \(T\) is completely continuous. 
	Furthermore, let $\mathcal{P}:= \left\{ u \in \mathcal{X}_+  : u(t) \geq \frac{1}{2}\|u\|_\infty, \, t \in \left[\frac{1}{2},1\right] \right\}$ be a new cone in \(\mathcal{X}\). 
	The nonnegativity of $f_1$ and \(f_2\) implies that $T(\mathcal{X}_+  \times \mathcal{P}) \subset \mathcal{X}_+  \times \mathcal{P}$. Indeed, note that \(T_j u(t) \ge 0\) for all \(t \in [0,1]\) ($j=1,2$). Moreover, since for each \(s \in [0,1]\) the function \(t \mapsto G_2(t,s)\) is nondecreasing, we have
	\begin{equation*}
		\begin{aligned}
			\min_{\frac{1}{2} \leq t \leq 1} T_{2}u(t) &= \int_{0}^{1} G_{2}\left(\frac{1}{2}, s\right) f_{2}(u_{1}(s), u_{2}(s)) \, ds \\
			&= \int_{0}^{1/2} s f_{2}(u_{1}(s), u_{2}(s)) \, ds + \int_{1/2}^{1} \frac{1}{2} f_{2}(u_{1}(s), u_{2}(s)) \, ds \\
			&\geq \frac{1}{2} \left[ \int_{0}^{1/2} s f_{2}(u_{1}(s), u_{2}(s)) \, ds + \int_{1/2}^{1} s f_{2}(u_{1}(s), u_{2}(s)) \, ds \right] \\
			&= \frac{1}{2} \int_{0}^{1} s f_{2}(u_{1}(s), u_{2}(s)) \, ds = \frac{T_{2}u(1)}{2} = \frac{\|T_{2}u\|_\infty}{2}.
		\end{aligned}
	\end{equation*}
\end{remark}

Now, in order to determine conditions for the existence of three solutions for the system \eqref{5.3}, we apply Theorem \ref{th_Kras_LW}, treating the second component in the expansive case. Note that, one can find conditions for this purpose by applying the classical Leggett-Williams theorem, however it will not ensure that at least two of the solutions have both components nontrivial. Observe also that, while the classical Leggett-Williams theorem imposes compression on the exterior boundary, Theorem \ref{th_Kras_LW} has opened the possibility of expansion in the second component, allowing the second nonlinearity to exhibit different behaviours.

\begin{theorem}
	\label{th5.1}
	Assume that there exists constants $a,c,d,r,R\in\mathbb{R}$ satisfying $0<d<a$, $2a<c$ and $0<2r<R$ such that 
	\begin{enumerate}[(a)]
		\item $f_1(x_1,x_2)\leq 2c$, for $0\leq x_1\leq c$, $0\leq x_2\leq R$;
		\item $f_1(x_1,x_2)< 2d$, for $0\leq x_1\leq d$, $0\leq x_2\leq R$;
		\item $f_1(x_1,x_2)> 4a$, for $a\leq x_1\leq 2a$, $\frac{1}{2}r\leq x_2\leq R$;
		\item $f_2(x_1,x_2)<2r$, for $0\leq x_1\leq c$, $0\leq x_2\leq r$;
		\item $f_2(x_1,x_2)>\frac{8}{3}R$, for $0\leq x_1\leq c$, $\frac{1}{2}R\leq x_2\leq R$.
	\end{enumerate}
Then, the system \eqref{5.3} has at least three solutions, at least two of them with both components nontrivial.
\end{theorem}

\noindent
{\bf Proof.} Note that the norm $\|\cdot\|_\infty:\mathcal{X}_+ \to[0,\infty)$ is a convex functional. Let $\alpha:\mathcal{X}_+ \to [0,\infty)$ be the concave funtional defined by
\[\alpha(u_1)=\min_{\frac{1}{2} \leq t \leq 1}u_1(t).\]  

Condition \emph{(a)} implies $T_1(\overline{\Omega}^{\|\cdot\|_\infty}_c\times (\overline{\mathcal{P}})_{r,R})\subset\overline{\Omega}^{\|\cdot\|_\infty}_c$. Indeed, let $u=(u_1,u_2)\in \overline{\Omega}^{\|\cdot\|_\infty}_c\times (\overline{\mathcal{P}})_{r,R}$. Then $0\leq u_1(t)\leq c$ and $0\leq u_2(t)\leq R$ for $t\in[0,1]$. Consequently, $0\leq T_1 u(t)$ for all $t\in[0,1]$, and 
\[\|T_1 u\|_\infty=T_1 u (1)=\int_{0}^{1} G_1(1,s)f_1(u_1(s),u_2(s))ds\leq2c\int_{0}^{1}sds=c.\]

We proceed to the verification of the conditions concerning the first component in Theorem \ref{th_Kras_LW}. Let $u_1\in\overline{\Omega}^{\|\cdot\|_\infty}_d$ and $u_2\in(\overline{\mathcal{P}})_{r,R}$. Then $0\leq u_1(t)\leq d$ and $0\leq u_2(t)\leq R$ for $t\in[0,1]$. By condition \emph{(b)} we obtain
\[\|T_1 u\|_\infty=T_1 u (1)=\int_{0}^{1} G_1(1,s)f_1(u_1(s),u_2(s))ds<2d\int_{0}^{1}sds=d.\]
Hence, condition \emph{(A)--(b)} of Theorem \ref{th_Kras_LW} is satisfied.

It is clear that there exists $u=(u_1,u_2)\in \mathcal{X}_+ [\alpha, \|\cdot\|_\infty, a, 2a]\times (\overline{\mathcal{P}})_{r,R}$ such that $\alpha(u_1)>a$. In addition, if $u=(u_1,u_2)\in \mathcal{X}_+ [\alpha, \|\cdot\|_\infty, a, 2a]\times (\overline{\mathcal{P}})_{r,R}$ with $\alpha(u_1)=a$, then $a\leq u_1(t)\leq 2a$ and $\frac{1}{2}r\leq u_2(t)\leq R$ for $t\in[1/2,1]$. Therefore, using condition \emph{(c)} we find
\[
\begin{aligned}
\alpha(T_1u)=\min_{\frac{1}{2} \leq t \leq 1}\int_{0}^{1}G_1(t,s)f_1(u_1(s),u_2(s))ds&=\int_{0}^{1}G_1\left(\frac{1}{2},s\right)f_1(u_1(s),u_2(s))ds\\
&\geq \int_{1/2}^{1}G_1\left(\frac{1}{2},s\right)f_1(u_1(s),u_2(s))ds>4a\int_{1/2}^{1}\frac{1}{2}ds=a.
\end{aligned}
\]
Thus, condition \emph{(A)--(a)} of Theorem \ref{th_Kras_LW} is satisfied for $b=2a$.

By reasoning as in Remark \ref{re5.1}, we have $\alpha(T_1u) \ge \frac{1}{2} \|T_1u\|_\infty.$ Hence, for any 
$u=(u_1,u_2) \in \mathcal{X}_+ [\alpha, \|\cdot\|_\infty, a, c] \times (\overline{\mathcal{P}})_{r,R}$ 
with $\alpha(u_1)=a$ and $\|T_1u\|_\infty > 2a$, it follows that $\alpha(T_1u) > a$. Then, condition \emph{(A)--(c)} of Theorem \ref{th_Kras_LW} holds.

We now proceed to verify condition \emph{(B)-(ii)} of Theorem \ref{th_Kras_LW}, which concerns the second component. Let $u=(u_1,u_2) \in \overline{\Omega}^{\|\cdot\|_\infty}_c \times (\overline{\mathcal{P}})_{r,R}$. If $\|u_2\|_\infty = r$, then $0 \le u_1(t) \le c$ and $0 \le u_2(t) \le r$ for all $t \in [0,1]$.  
By condition \emph{(d)}, we have
\[
\|T_2u\|_\infty = T_2u(1) = \int_0^1 G_2(1,s) f_2(u_1(s),u_2(s)) \, ds 
= \int_0^1 s f_2(u_1(s),u_2(s)) \, ds < 2r \int_0^1 s \, ds = r = \|u_2\|_\infty.
\]

Similarly, if $\|u_2\|_\infty = R$, then $0 \le u_1(t) \le c$ and $\frac{1}{2}R \le u_2(t) \le R$ for all $t \in [1/2,1]$.  
Condition \emph{(e)} ensures
\[
\begin{aligned}
	\|T_2u\|_\infty = T_2u(1) = \int_0^1 G_2(1,s) f_2(u_1(s),u_2(s)) \, ds &\ge \int_{1/2}^1 G_2(1,s) f_2(u_1(s),u_2(s)) \, ds \\
	&= \int_{1/2}^1 s f_2(u_1(s),u_2(s)) \, ds > \frac{8}{3} R \int_{1/2}^1 s \, ds = R = \|u_2\|_\infty.
\end{aligned}
\]

As a result, Theorem \ref{th_Kras_LW} guarantees the existence of at least three fixed points for the operator $T$ in 
$\overline{\Omega}^{\|\cdot\|_\infty}_c \times (\overline{\mathcal{P}})_{r,R}$, two of which are coexistence fixed points.  
Accordingly, system \eqref{5.3} has three solutions in this set, including at least two with both components nontrivial, as desired. \qed

\begin{example}
The system 
\begin{equation}
	\label{5.5}
	\left\{
	\begin{aligned}
		& x''(t) + \frac{1}{2}+ 5\varphi(x(t))\psi(y(t)) = 0, \quad t\in[0,1], \\
		& y''(t) + e^{y^2(t)/32}+\frac{1}{10}\cos(\pi x(t))= 0, \quad t\in[0,1],\\
		& x(0) = x'(1) = 0 = y(0) = y'(1), 
	\end{aligned}
	\right.
\end{equation}
where
\begin{equation}
	\label{phi_psi}
\varphi(z) =
\begin{cases}
	0, & 0 \le z \le 1/2,\\
	2z-1, & 1/2 < z \le 1,\\
	1, & z>1,
\end{cases}
\qquad \text{and} \qquad
\psi(z) =
\begin{cases}
	z, & 0 \le z \le 1,\\
	1, & z>1,
\end{cases}
\end{equation}
satisfies the assumptions of Theorem \ref{th5.1} with the constants $a=1$, $c=5$, $d=1/2$, $r=2$, and $R=5$. 
Moreover, it is readily seen that, for any solution \((x,y)\) of \eqref{5.5}, the component \(x\) cannot be identically zero (i.e, the function $x(t)$ cannot vanish for all $t\in[0,1]$). 
Consequently, system \eqref{5.5} admits at least three solutions, each with both components nontrivial.
\end{example}

Conditions on the nonlinearities that ensure the existence of at least nine solutions for the system \eqref{5.3} are detailed in the following result, which is derived from the application of Theorem \ref{th2}.
\begin{theorem}
		\label{th5.2}
		Assume that there exists constants $a_j,c_j,d_j\in\mathbb{R}$ ($j=1,2$) satisfying $0<d_j<a_j$ and $2a_j\le c_j$ such that the following conditions hold for each $j\in\{1,2\}$ 
		\begin{enumerate}[(a)]
			\item $ f_j(x_1,x_2)\leq 2c_j$, for $0\leq x_1\leq c_1$, $0\leq x_2\leq c_2$;
			\item $f_j(x_1,x_2)< 2d_j$, for $0\leq x_j\leq d_j$, $0\leq x_i\leq c_i$, $(i\neq j)$;
			\item $f_j(x_1,x_2)> 4a_j$, for $a_j\leq x_j\leq 2a_j$, $0\leq x_i\leq c_i$, $(i\neq j)$;
		\end{enumerate}
		Then, the system \eqref{5.3} has at least nine solutions, at least four of them with both components nontrivial.
\end{theorem}

\noindent
{\bf Proof.} As in the proof of Theorem~\ref{th5.1}, we use $\|\cdot\|_\infty$ and $\alpha$
as the convex and concave functionals, respectively. For $j=1,2$, we set $\overline{\Omega}_{c_j}^{\|\cdot\|_\infty}
=\{u\in\mathcal{X}_+:\ \|u\|_\infty\le c_j\}$.

Fix $j\in\{1,2\}$. By condition~\emph{(a)}, $T_j\bigl(\overline{\Omega}_{c_1}^{\|\cdot\|_\infty}
\times \overline{\Omega}_{c_2}^{\|\cdot\|_\infty}\bigr)
\subset \overline{\Omega}_{c_j}^{\|\cdot\|_\infty}$. Indeed, for $u=(u_1,u_2)\in \overline{\Omega}_{c_1}^{\|\cdot\|_\infty}
\times \overline{\Omega}_{c_2}^{\|\cdot\|_\infty}$ we have
$0\le u_i(t)\le c_i$ for $t\in[0,1]$, $i=1,2$, and hence
$0\le T_j u(t)$ for all $t\in[0,1]$. Moreover,
\[
\|T_j u\|_\infty
= T_j u(1)
= \int_0^1 G_j(1,s) f_j(u_1(s),u_2(s))\,ds
\le 2c_j \int_0^1 s\,ds
= c_j.
\]

Next, let $u_j\in \overline{\Omega}^{\|\cdot\|_\infty}_{d_j}$ and
$u_i\in \overline{\Omega}^{\|\cdot\|_\infty}_{c_i}$, with $i\neq j$.
Then $0\le u_j(t)\le d_j$ and $0\le u_i(t)\le c_i$ for all $t\in[0,1]$.
Condition~\emph{(b)} yields
\[
\|T_j u\|_\infty
= T_j u(1)
= \int_0^1 G_j(1,s) f_j(u_1(s),u_2(s))\,ds
< 2d_j \int_0^1 s\,ds
= d_j,
\]
so that condition~\emph{(A)--(b)} of Theorem~\ref{th2} holds for $j$.

It is clear that there exists
$u_j\in \mathcal{X}_+[\alpha,\|\cdot\|_\infty,a_j,2a_j]$
such that $\alpha(u_j)>a_j$. Moreover, let
$u_j\in \mathcal{X}_+[\alpha,\|\cdot\|_\infty,a_j,2a_j]$ and
$u_i\in \overline{\Omega}^{\|\cdot\|_\infty}_{c_i}$, $i\neq j$, with
$\alpha(u_j)=a_j$. Then
$a_j\le u_j(t)\le 2a_j$ and $0\le u_i(t)\le c_i$ for
$t\in[1/2,1]$. By condition~\emph{(c)}, we obtain
\[
\begin{aligned}
	\alpha(T_j u)
	= \min_{1/2\le t\le 1}
	\int_0^1 G_j(t,s) f_j(u_1(s),u_2(s))\,ds& = \int_0^1 G_j\!\left(\tfrac12,s\right)
	f_j(u_1(s),u_2(s))\,ds \\
	&\ge \int_{1/2}^1 G_j\!\left(\tfrac12,s\right)
	f_j(u_1(s),u_2(s))\,ds
	> 4a_j \int_{1/2}^1 \tfrac12\,ds
	= a_j.
\end{aligned}
\]
Thus, condition~\emph{(A)--(a)} of Theorem~\ref{th2} is satisfied for $j$
by choosing $b_j=2a_j$.

Finally, since $\alpha(T_j u)\ge \tfrac12\|T_j u\|_\infty$, for
$u_j\in \mathcal{X}_+[\alpha,\|\cdot\|_\infty,a_j,c_j]$ and
$u_i\in \overline{\Omega}^{\|\cdot\|_\infty}_{c_i}$, $i\neq j$, with
$\alpha(u_j)=a_j$ and $\|T_j u\|_\infty>2a_j$, we infer that
$\alpha(T_j u)>a_j$. Hence, condition~\emph{(A)--(c)} of
Theorem~\ref{th2} also holds for $j$.

Therefore, the operator $T$ satisfies all the assumptions of
Theorem~\ref{th2}, which ensures the existence of at least nine solutions
in $\overline{\Omega}_{c_1}^{\|\cdot\|_\infty}
\times \overline{\Omega}_{c_2}^{\|\cdot\|_\infty}$, four of them being
coexistence fixed points. Consequently, system~\eqref{5.3} admits at
least nine solutions, including at least four with both components
nontrivial. \qed

\begin{example}
		\label{symmetric_systems}
Consider the symmetric system of the form
\begin{equation}
	\label{symmetric}
	\left\{
	\begin{aligned}
		& x''(t) + f(x(t), y(t)) = 0, \quad t\in[0,1], \\
		& y''(t) + f(y(t), x(t)) = 0, \quad t\in[0,1],\\
		& x(0) = x'(1) = 0 = y(0) = y'(1), 
	\end{aligned}
	\right.
\end{equation}
where $f:[0,\infty)^2\to [0,\infty)$ is a continuous function. Theorem \ref{th5.2} ensures that the system \eqref{symmetric} admits nine distinct solutions, four of which have both components nontrivial, provided that there exist $a,c,d\in\mathbb{R}$ such that $0<d<a$ and $2a<c$, and the following conditions hold:
\begin{enumerate}[(a)]
	\item $ f(x,y)\leq 2c$, for $0\leq x\leq c$, $0\leq y\leq c$;
	\item $f(x,y)< 2d$, for $0\leq x\leq d$, $0\leq y\leq c$;
	\item $f(x,y)> 4a$, for $a\leq x\leq 2a$, $0\leq y\leq c$.
\end{enumerate}

Every function of the form
\begin{equation}
	\label{example}
	f(x,y)=Lg(x)h(y),
\end{equation}
where $h:[0,\infty)\to (0,\infty)$ is non-increasing and $g:[0,\infty)\to [0,\infty)$ is non-decreasing and fulfills
\[
\lim_{x\to 0^+}\frac{g(x)}{x}=0
\quad \text{and} \quad
\lim_{x\to +\infty} \frac{g(x)}{x}=0,
\]
satisfies the above conditions provided that $L$ is sufficiently large.

Indeed, since $g$ is non-decreasing and $\lim_{x\to +\infty} g(x)/x=0$, there exists $c>0$ large enough such that
\[\dfrac{g(c)}{c}\leq \dfrac{2}{L h(0)} \]
and thus
\[g(x)\leq g(c)\leq \dfrac{2}{L h(0)}c \quad \text{for } 0\leq x\leq c. \]
By the monotonicity of $h$, it follows that
\[L g(x)h(y)\leq \dfrac{2}{h(0)}h(y) c\leq 2c \quad \text{for } 0\leq x\leq c, \, 0\leq y, \]
which yields condition (a).

Now, suppose that the inequality
\begin{equation}
	\label{L_cond}
	L>\frac{4a}{g(a) h(c)}
\end{equation}
holds for some $a\in(0,c/2)$. Then $L g(a)h(c)>4a$ and since $g$ is non-decreasing, we have
\[
L g(x) h(c)>4a, \quad \text{for } a\leq x\leq 2a.
\]
Moreover, since $h(y)\geq h(c)$ for all $0\leq y\leq c$, it follows that
\[
L g(x) h(y)>4a, \quad \text{for } a\leq x\leq 2a, \, 0\leq y\leq c,
\]
which implies condition (c).

On the other hand, since
\[
\lim_{x\to 0^+} \frac{g(x)}{x}=0,
\]
there exists $d\in(0,a)$ such that
\[
L g(x)<\frac{2}{h(0)}x, \quad \text{for } 0\leq x\leq d.
\]
Using that $h$ is non-increasing, we obtain
\[
Lg(x)h(y)<2d, \quad \text{for } 0\leq x\leq d, \, 0\leq y,
\]
and this implies condition (b).

Typical examples of functions $g$ satisfying the above conditions include 
\[g(x)=\dfrac{x^2}{1+x^2}, \quad g(x)=\ln(1+x^2), \quad g(x)=\min\{x^p,x^q \} \quad \text{ and } \quad g(x)=1-e^{-x^p}\]
with $p>1$ and $0<q<1$.
Similarly, functions $h$ of the form $h(y)=(1+y)^{-k}$ or $ h(y)=e^{-k y}$ with $k>0$ are admissible choices in \eqref{example}.
\end{example}

In some cases, it is possible to argue that the nine solutions provided by Theorem \ref{th5.2} have both components nontrivial. This is illustrated by the following example.

\begin{example}
			Consider the following symmetric system
		\begin{equation}
			\label{5.61}
			\left\{
			\begin{aligned}
				& x''(t) + \frac{9}{2} + 5\varphi(x(t))\psi(y(t))-4\phi(x(t)) = 0, \quad t \in [0, 1], \\
				& y''(t) + \frac{9}{2} + 5\varphi(y(t))\psi(x(t))-4\phi(y(t)) = 0, \quad t \in [0, 1], \\
				& x(0) = x'(1) = 0, \quad y(0) = y'(1) = 0, 
			\end{aligned}
			\right.
		\end{equation}
		where the nonlinearities $\varphi$ and $\psi$ are defined as in \eqref{phi_psi} and
		\[\phi(z) =
		\begin{cases}
			1, & 0 \le z \le 1/2,\\
			2-2z, & 1/2 < z \le 1,\\
			0, & z>1.
		\end{cases}\]
		
		 By choosing the constants $a_j=1$, $c_j=5$, and $d_j=1/2$ for $j=1,2$, it can be verified that the conditions (a)--(c) of Theorem~\ref{th5.2} are satisfied. Consequently, the system \eqref{5.61} possesses at least nine distinct solutions. Furthermore, due to the presence of the positive constant term $9/2$ in the differential equations, it is evident that every component of any solution is strictly positive. Therefore, all nine solutions have both nontrivial components.
\end{example}

Next, we turn our attention to the following system of differential equations
\begin{equation}
	\label{5.6}
	\left\{
	\begin{aligned}
		& \beta_1 u_1''(t) - u_1'(t) + f_1(u_1(t), u_2(t)) = 0, \quad t\in[0,1], \\
		& \beta_2 u_2''(t) - u_2'(t) + f_2(u_1(t), u_2(t)) = 0, \quad t\in[0,1],\\
		& \beta_1 u_1'(0) - u_1(0) = 0, \quad u_1'(1) = 0, \\
		& \beta_2 u_2'(0) - u_2(0) = 0, \quad u_2'(1) = 0,
	\end{aligned}
	\right.
\end{equation}
where $\beta_1,\beta_2>0$ and $f_1,f_2:[0,\infty)^2\to\mathbb{R}$ are continuous functions.

Systems of the form \eqref{5.6} arise naturally in reaction-convection-diffusion models, with applications to chemical reactor dynamics and biological mass transport, among others \cite{Leggett-Williams, Deim, varma}.

The solutions of system \eqref{5.6} are in one-to-one correspondence with the fixed points of an operator of the form given in \eqref{5.4}. For each $j \in \{1,2\}$, the associated Green's functions in this case (see, e.g. \cite{infante2, Leggett-Williams}) are given by
\begin{equation}
	\label{green2}
	G_j(t,s)=
	\begin{cases}
		e^{\frac{t-s}{\beta_j}}, & 0 \le t \le s \le 1,\\[4pt]
		1, & 0 \le s \le t \le 1.
	\end{cases}
\end{equation}
Hereafter, we shall again denote this operator by $T$.

Note that, as before, for each $j \in \{1,2\}$ and each fixed $s \in [0,1]$, the mapping $t \mapsto G_j(t,s)$, with $G_j$ defined by \eqref{green2}, is nondecreasing on $[0,1]$.

In the following theorem, we establish sufficient conditions for the existence of at least four distinct solutions for system \eqref{5.6} by applying Corollary \ref{coro4.1} to the operator $T$. 

\begin{theorem}
	\label{th5.3}
Assume that there exist constants $a_j, c_j, d_j \in \mathbb{R}$, with $0 < d_j < a_j$ and 
$a_j e^{1/\beta_j} < c_j$ ($j = 1,2$), such that the following conditions hold for each $j \in \{1,2\}$:
\begin{enumerate}[(a)]
	\item $0\leq f_j(x_1,x_2)$, for $0\leq x_1\leq c_1$, $0\leq x_2\leq c_2$;
	\item  $f_j(x_1,x_2)<d_j$, for $0\leq x_j\leq d_j$, $0\leq x_i\leq c_i$ ($j\neq i$);
	\item $f_j(x_1,x_2)>a_j(\beta_j-\beta_je^{-1/\beta_j})^{-1}$, for $a_j\leq x_j\leq c_j$, $0\leq x_i\leq c_i$ ($j\neq i$).
\end{enumerate}
	Then, the system \eqref{5.5} has at least four solutions, at least one of them with both components nontrivial.
\end{theorem}

\noindent
{\bf Proof.} Again, we will be working with the norm $\|\cdot\|_\infty:\mathcal{X}_+ \to[0,\infty)$ as a convex functional. For each $j\in\{1,2\}$ we introduce $\alpha_j:\mathcal{X}_+ \to [0,\infty)$ the concave funtional defined by
\[\alpha_j(u_j)=\min_{0 \leq t \leq 1}u_j(t).\]

Observe that, from condition \emph{(a)}, one has $T_j\bigl(\overline{\Omega}^{\|\cdot\|_\infty}_{c_1}\times
\overline{\Omega}^{\|\cdot\|_\infty}_{c_2}\bigr)
\subset \mathcal{X}_+$  for both \(j\in\{1,2\}\), where $\overline{\Omega}^{\|\cdot\|_\infty}_{c_j} :=\{u_j\in\mathcal{X}_+ : \|u_j\|_\infty \le c_j\}$. 

Fix \(j\in\{1,2\}\). Let \(u_j\in \overline{\Omega}^{\|\cdot\|_\infty}_{d_j}\) and
\(u_i\in \overline{\Omega}^{\|\cdot\|_\infty}_{c_i}\), with \(i\neq j\).
Then $0 \le u_j(t) \le d_j$ and $0 \le u_i(t) \le c_i$ for all $t\in[0,1]$. Hence, by condition \emph{(b)} for $j$, it follows that
\[
\|T_j u\|_\infty
= T_j u(1)
= \int_0^1 G_j(1,s)\, f_j\bigl(u_1(s),u_2(s)\bigr)\, ds
< d_j \int_0^1 G_j(1,s)\, ds
= d_j.
\]
Hence, condition \emph{(b)} of Corollary \ref{coro4.1} is satisfied for both $j=1,2$.

Let $u_j\in\mathcal{X}_+[\alpha_j,\|\cdot\|_\infty,a_j,c_j]$ with $\alpha_j(u_j)=a_j$ and $u_i\in\overline{\Omega}^{\|\cdot\|_\infty}_{c_i}$, $i\neq j$. Then $a_j\leq u_j(t)\leq c_j$ and $0\leq u_i(t)\leq c_i$ for $t\in[0,1]$. Therefore, using condition \emph{(c)} for $j$ we find
\[
\begin{aligned}
\alpha_j(T_j u)=T_j u(0)&=\int_{0}^{1} G_j(0,s)f_j(u_1(s),u_2(s))ds\\
&=\int_{0}^{1}e^{-\frac{s}{\beta_j}}f_j(u_1(s),u_2(s))ds>a_j(\beta_j-\beta_je^{-\frac{1}{\beta_j}})^{-1}\int_{0}^{1}e^{-\frac{s}{\beta_j}}ds=a_j.
\end{aligned}
\]
Thus, condition \emph{(a)} of Corollary \ref{coro4.1} holds for both $j=1,2$.

Finally, assume that $u_j\in\mathcal{X}_+[\alpha_j,\|\cdot\|_\infty,a_j,c_j]$ with $\alpha_j(u_j)=a_j$ and $u_i\in\overline{\Omega}^{\|\cdot\|_\infty}_{c_i}$, $i\neq j$, such that $\|T_j u\|_\infty>c_j$. We have
\begin{equation}
	\label{cond_ig}
\begin{aligned}
	\alpha_j(T_j u)=T_j u(0)&=\int_{0}^{1}e^{-\frac{s}{\beta_j}}f_j(u_1(s),u_2(s))ds\\
	&\ge e^{-\frac{1}{\beta_j}}\int_{0}^{1} f_j(u_1(s),u_2(s))ds=e^{-\frac{1}{\beta_j}} T_j u (1)=e^{-\frac{1}{\beta_j}}\|T_j u\|_\infty>\frac{a_j}{c_j} \|T_j u\|_\infty.
\end{aligned}
\end{equation}
 Hence, condition \emph{(c)} of Corollary~\ref{coro4.1} is satisfied for both
 \(j\in\{1,2\}\). 
 
Therefore, Corollary~\ref{coro4.1} guarantees the existence of at least four fixed points of the operator \(T\) in $\overline{\Omega}^{\|\cdot\|_\infty}_{c_1}\times
\overline{\Omega}^{\|\cdot\|_\infty}_{c_2},$ at least one of which is a coexistence fixed point. Consequently, system~\eqref{5.6} admits at least four solutions, one of them having both components nontrivial. \qed

\begin{remark}
		\label{re5.2}
	Note that, whenever the nonlinearities are strictly positive on $[0,c_1] \times [0,c_2]$, we can choose $a_j e^{1/\beta_j} = c_j$ for $j=1,2$ in Theorem \ref{th5.3}, since in \eqref{cond_ig} the first inequality is then strict and strictness is no longer required in the last inequality.
\end{remark}

We now apply Theorem \ref{th5.3} to the following system with parameters
\begin{equation}
	\label{5.9}
	\left\{
	\begin{aligned}
		& \beta_1 u_1''(t) - u_1'(t)
		+ p_1 \bigl(q_1 - u_2(t)\bigr)\, \exp\!\left(\frac{-k_1}{1+u_1(t)}\right) = 0,
		\quad t\in[0,1], \\
		& \beta_2 u_2''(t) - u_2'(t)
		+ p_2 \bigl(q_2 - u_1(t)\bigr)\, \exp\!\left(\frac{-k_2}{1+u_2(t)}\right) = 0,
		\quad t\in[0,1], \\
		& \beta_1 u_1'(0) - u_1(0) = 0,
		\qquad u_1'(1) = 0, \\
		& \beta_2 u_2'(0) - u_2(0) = 0,
		\qquad u_2'(1) = 0,
	\end{aligned}
	\right.
\end{equation}
where \(\beta_1,\beta_2,p_1,p_2,q_1,q_2,k_1,k_2>0\). This system is of the form \eqref{5.6}, where the nonlinearities are given by
\begin{equation*}
	f_1(u_1,u_2) = p_1 (q_1 - u_2)\,  \exp\!\left(\frac{-k_1}{1+u_1}\right)
	\quad \text{ and } \quad
	f_2(u_1,u_2) = p_2 (q_2 - u_1)\,  \exp\!\left(\frac{-k_2}{1+u_2}\right).
\end{equation*}

For $k\in\mathbb{R}_+$, we define the function $g_k:(0,\infty)\to \mathbb{R}$ given by
\begin{equation}
	g_k(z)=\frac{1}{z}\exp\left(\frac{-k}{1+z}\right),
\end{equation}
together with the following numbers
\[s(k)=\frac{(k-2)-\sqrt{k(k-4)}}{2} \quad \text{ and } \quad \tilde{s}(k)=\frac{(k-2)+\sqrt{k(k-4)}}{2}. \]
Understanding the increasing and decreasing behavior of \(g\) is essential for the subsequent analysis. In this regard, we obtain the following lemma.
\begin{lemma}
	\label{le5.1}
	For $k\in\mathbb{R}_+$, let \(g_k:(0,\infty)\to\mathbb{R}\) be defined as above. Then $g_k$ is nonincreasing on $(0,\infty)$ if $k\le 4$. Otherwise, if \(k>4\), \(g_k\) is decreasing on \((0,s(k))\cup(\tilde{s}(k),\infty)\), increasing on \((s(k),\tilde{s}(k))\), attaining a relative minimum and maximun at \(s(k)\) and \(\tilde{s}(k)\), respectively.
\end{lemma}

\begin{figure}[h]
	\centering
		\hspace{-1.1cm}\includegraphics[scale=0.8]{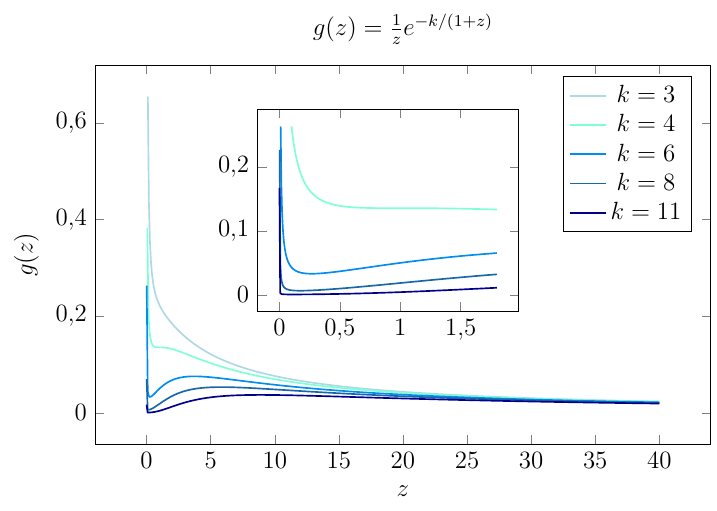}
		\vspace{-0.5cm}
	\caption{Graphical illustration of the behavior of \(g_k\) for selected values of \(k\): \(k=3,4,6,8,11\).}
	\label{fig3}
\end{figure}

\begin{proposition}
	\label{prop5.1}
For system~\eqref{5.9}, let \(\beta_1, \beta_2 > 0\), and assume that
\begin{equation}
	\label{5.11}
	\exp\bigl(-\sqrt{k_1(k_1-4)}\,\bigr)
	< \frac{s(k_1)}{\tilde{s}(k_1)}\,\frac{k_2-1}{k_2}, 
	\qquad
	\exp\bigl(-\sqrt{k_2(k_2-4)}\,\bigr)
	< \frac{s(k_2)}{\tilde{s}(k_2)}\,\frac{k_1-1}{k_1},
\end{equation}
together with $q_1 = r_2 \, \tilde{s}(k_2)\, e^{1/\beta_2},$ and $q_2 = r_1 \, \tilde{s}(k_1)\, e^{1/\beta_1},$ where \(r_1 \ge k_1\) and \(r_2 \ge k_2\). 

Then, there exist \(m_1,m_2\in\mathbb{R}_+\) such that, setting $p_1 = m_2 \, e^{-1/\beta_2}$ and $p_2 = m_1 \, e^{-1/\beta_1},$ the following inequalities hold:
\begin{equation}
	\label{5.12}
\frac{f_1(s(k_1),0)}{s(k_1)} < 1, \qquad
\frac{f_1(\tilde{s}(k_1),\tilde{s}(k_2)e^{1/\beta_2})}{\tilde{s}(k_1)} > 1,
\end{equation}
\begin{equation}
	\label{5.13}
	\frac{f_2(0,s(k_2))}{s(k_2)} < 1, \qquad
	\frac{f_2(\tilde{s}(k_1)e^{1/\beta_1}, \tilde{s}(k_2))}{\tilde{s}(k_2)} > 1.
\end{equation}

Moreover, the inequalities \eqref{5.12} and \eqref{5.13} are satisfied for any pair $(m_1, m_2) \in \mathbb{R}_+^2$ fulfilling
\begin{equation}
	\label{5.14}
	\frac{\tilde{s}(k_2)}{(r_1 - 1) \tilde{s}(k_1)} \exp \left( \frac{k_2}{1 + \tilde{s}(k_2)} \right) < m_1 < \frac{s(k_2)}{r_1 \tilde{s}(k_1)} \exp \left( \frac{k_2}{1 + s(k_2)} \right),
\end{equation}
\begin{equation}
	\label{5.15}
		\frac{\tilde{s}(k_1)}{(r_2 - 1) \tilde{s}(k_2)} \exp \left( \frac{k_1}{1 + \tilde{s}(k_1)} \right) < m_2 < \frac{s(k_1)}{r_2 \tilde{s}(k_2)} \exp \left( \frac{k_1}{1 + s(k_1)} \right).
\end{equation}
\end{proposition}

\noindent
{\bf Proof.} It should be noted that the inequalities in \eqref{5.11} are meaningful only for \(k_1, k_2 > 4\); accordingly, we assume this without further mention. 

Now fix \(j\in\{1,2\}\), and let \(x_i \in [0, \tilde{s}(k_i) e^{1/\beta_i}]\) with \(i\neq j\). 
Then, by Lemma~\ref{le5.1}, as clearly $q_j-x_i>0$, the function
\[
x_j \mapsto \frac{f_j(x)}{x_j} = p_j (q_j - x_i)\, \frac{1}{x_j} \exp\!\left(\frac{-k_j}{1+x_j}\right)
\]
is decreasing on \((0, s(k_j)) \cup (\tilde{s}(k_j), \infty)\), increasing on \((s(k_j), \tilde{s}(k_j))\), and attains a relative minimum and maximum at \(s(k_j)\) and \(\tilde{s}(k_j)\), respectively.

For \eqref{5.12} to be satisfied, it is required that
\[\frac{m_2 r_2 \tilde{s}(k_2)}{s(k_1)}\exp  \left( \frac{-k_1}{1 +s(k_1)} \right)<1\quad  \text{ and } \quad \frac{m_2 (r_2-1) \tilde{s}(k_2)}{\tilde{s}(k_1)}\exp  \left( \frac{-k_1}{1 + \tilde{s}(k_1)} \right)>1.\]

Therefore, it is sufficient to select $m_2$ fulfilling
\begin{equation*}
	\frac{\tilde{s}(k_1)}{(r_2 - 1) \tilde{s}(k_2)} \exp \left( \frac{k_1}{1 + \tilde{s}(k_1)} \right) < m_2 < \frac{s(k_1)}{r_2 \tilde{s}(k_2)} \exp \left( \frac{k_1}{1 + s(k_1)} \right).
\end{equation*}

The existence of such $m_2$ is guaranteed provided that 
\begin{equation*}
	\frac{\tilde{s}(k_1)}{(r_2 - 1) \tilde{s}(k_2)} \exp \left( \frac{k_1}{1 + \tilde{s}(k_1)} \right) < \frac{s(k_1)}{r_2 \tilde{s}(k_2)} \exp \left( \frac{k_1}{1 + s(k_1)} \right),
\end{equation*}
or, equivalently,
\begin{equation*}
	\exp\left(-\sqrt{k_1(k_1-4)}\right) < \frac{s(k_1)}{\tilde{s}(k_1)} \frac{r_2 - 1}{r_2}.
\end{equation*}
The latter estimate holds as a consequence of the first inequality in \eqref{5.11} together with the condition $k_2 \leq r_2$.

An analogous argument, accounting for the second inequality in \eqref{5.11} and $k_1 \leq r_1$, yields the existence of a constant $m_1$ fulfilling \eqref{5.13}. \qed

\begin{remark}
	Since we assume \(k_1,k_2>4\) so that \eqref{5.11} is meaningful, we immediately obtain
	\[
	\frac{k_j}{k_j-1}<\frac{4}{3}, \qquad j=1,2.
	\]
	Hence, inequality \eqref{5.11} is satisfied whenever
	\[
	\frac{s(k_j)}{\tilde{s}(k_j)}e^{\sqrt{k_j(k_j-4)}}>\frac{4}{3}, \qquad j=1,2.
	\]
	
	Consider the function \(h:[4,\infty)\to\mathbb{R}\) defined by
	\[
	h(z)=\frac{s(z)}{\tilde{s}(z)}e^{\sqrt{z(z-4)}}-\frac{4}{3}.
	\]
	One readily verifies that \(h\) is monotone increasing on its domain. Moreover,
	numerical computations show that \(h\) admits a zero \(z_0<5\). As a consequence,
	\(h(z)>0\) for all \(z>z_0\), and therefore it is sufficient to choose
	\(k_1,k_2\) sufficiently large (for instance, \(k_1,k_2>5\)) in order to ensure
	that \eqref{5.11} holds.
\end{remark}

\begin{figure}[h]
	\centering
	\hspace{-1.1cm}\includegraphics[scale=0.8]{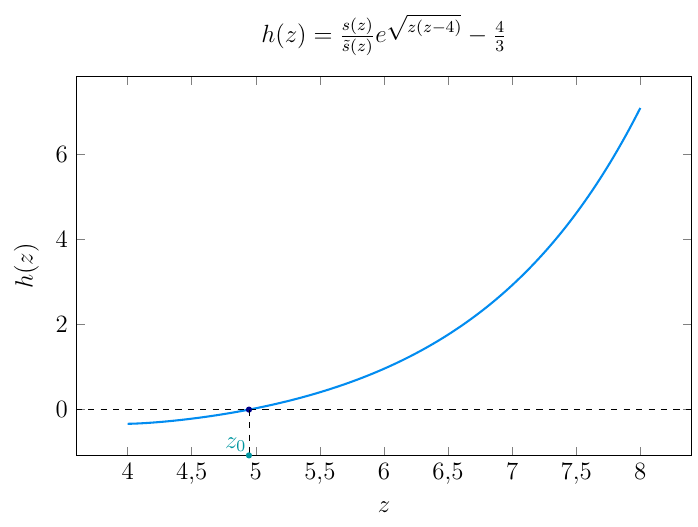}
	\vspace{-0.5cm}
	\caption{Graphical illustration of \(h\) in $[4,8]$.}
	\label{fig4}
\end{figure}

\begin{theorem}
	\label{last_th}
Under the assumptions of Proposition~\ref{prop5.1}, set $p_1 = m_2 e^{-1/\beta_2}$ and $p_2 = m_1 e^{-1/\beta_1}$, where
$m_1, m_2 \in \mathbb{R}$ satisfy \eqref{5.14} and \eqref{5.15}, respectively.
Then system~\eqref{5.9} admits at least four solutions, each with
both components nontrivial, for all $\beta_1$ and $\beta_2$ satisfying
\begin{equation}
	\label{5.16}
	\beta_1 - \beta_1 e^{-1/\beta_1}
	> \frac{\tilde{s}(k_1)}{f_1\bigl(\tilde{s}(k_1), \tilde{s}(k_2) e^{1/\beta_2}\bigr)}
	\quad \text{and} \quad
	\beta_2 - \beta_2 e^{-1/\beta_2}
	> \frac{\tilde{s}(k_2)}{f_2\bigl(\tilde{s}(k_1) e^{1/\beta_1}, \tilde{s}(k_2)\bigr)}.
\end{equation}
\end{theorem}

\noindent
{\bf Proof.} Under the stated assumptions, Theorem~\ref{th5.3} is applicable with $d_j = s(k_j)$, $a_j = \tilde{s}(k_j)$, and $c_j = \tilde{s}(k_j) e^{1/\beta_j}$ for $j=1,2$, which ensures the existence of the desired four solutions. We check here conditions \emph{(a)--(b)--(c)} for $j=1$; the case $j=2$ is symmetric.

Recall that we are implicitly assuming $k_1,k_2>4$. As  $q_1= r_2 \, \tilde{s}(k_2)\, e^{1/\beta_2}$ with $k_2\leq r_2$ and $p_1>0$ we have
\[0\leq f_1(x_1,x_2) \quad \text{ for } 0\leq x_1\leq \tilde{s}(k_1)e^{1/\beta_1} \, 0\leq x_2\leq \tilde{s}(k_2)e^{1/\beta_2}.\]
Thus, condition \emph{(a)} in Theorem \ref{th5.3} is satisfied.

Note that, for any fixed $x_2 \in[0,\tilde{s}(k_2)e^{1/\beta_2}]$, $f_1(\cdot, x_2)$ is an increasing function on $[0,\infty)$, while for any fixed $x_1 \geq 0$, $f_1(x_1, \cdot)$ is a decreasing function on $[0,\infty)$. 

Under the assumptions,
\[
\frac{f_1(s(k_1),0)}{s(k_1)} < 1,
\]
and by the monotonicity of $f_1$ in each variable, we deduce
\[
f_1(x_1,x_2) < s(k_1), \quad 
0 \leq x_1 \leq s(k_1), \; 0 \leq x_2 \leq \tilde{s}(k_2) e^{1/\beta_2},
\]
which corresponds with condition \emph{(b)} in Theorem \ref{th5.3}.

Finally, let $\beta_1$ satisfy the first inequality in \eqref{5.16}. Such a $\beta_1$ exists since the right-hand side is always less than one and $\lim_{t \to \infty} t - te^{-1/t} = 1.$ This inequality corresponds to $f_1\bigl(\tilde{s}(k_1), \tilde{s}(k_2) e^{1/\beta_2}\bigr) >\tilde{s}(k_1) (\beta_1 - \beta_1 e^{-1/\beta_1})^{-1}$. Once again, the monotonicity of $f_1$ ensures that
\[
f_1(x_1,x_2) > \tilde{s}(k_1)(\beta_1 - \beta_1 e^{-1/\beta_1})^{-1}, \quad 
\tilde{s}(k_1) \leq x_1 \leq \tilde{s}(k_1)e^{1/\beta_1}, \; 0 \leq x_2 \leq \tilde{s}(k_2) e^{1/\beta_2}.
\]
Hence condition \emph{(c)} in Theorem \eqref{th5.3} holds.

Finally, suppose that $(\bar{x},\bar{y})$ is a solution of system \eqref{5.9} with $\bar{x}(t) = 0, \, t\in[0,1]$. Then the first equation forces $\bar{y}(t) = q_1,\, t\in[0,1]$,  which does not satisfy the second equation. A symmetric argument shows that the second component $\bar{y}$ cannot be identically zero either. Therefore, both components are necessarily nontrivial.  \qed

We conclude with an example illustrating Theorem \ref{last_th}.

\begin{example}
	The following system 
	\begin{equation}
		\begin{cases}
			\beta_1 x''(t) - x'(t) + e^{-1/\beta_2} \Big( 10(4 + \sqrt{15})e^{1/\beta_2} - y(t) \Big) \exp\Big( \frac{-8}{1 + x(t)} \Big) = 0,\quad t\in[0,1],\\[2mm]
			\beta_2 y''(t) - y'(t) + 3 e^{-1/\beta_1} \Big( 8(3 + \sqrt{8})e^{1/\beta_1} - x(t) \Big) \exp\Big( \frac{-10}{1 + y(t)} \Big) = 0,\quad t\in[0,1],\\[1mm]
			\beta_1 x'(0) - x(0) = 0, \quad x'(1) = 0,\\
			\beta_2 y'(0) - y(0) = 0, \quad y'(1) = 0
		\end{cases}
	\end{equation}
	has at least four solutions, each with both components nontrivial, provided that
	\[
	\beta_1 - \beta_1 e^{-1/\beta_1} > \frac{3+2\sqrt{2}}{9(4 + \sqrt{15})}\exp\left(\frac{4}{2+\sqrt{2}}\right), 
	\qquad
	\beta_2 - \beta_2 e^{-1/\beta_2} >\frac{4+\sqrt{15}}{21(3 + 2\sqrt{2})}\exp\left(\frac{10}{5+\sqrt{15}}\right).
	\]
	
	In this example, we take $	k_1 = 8,\, k_2 = 10,\,
	q_1 = 10 \tilde{s}(10)e^{1/\beta_2}, \, q_2 = 8 \tilde{s}(8)e^{1/\beta_1}, \, 
	p_1 = e^{-1/\beta_2} $ and $ p_2 = 3 e^{-1/\beta_1}.$
\end{example}

\section*{Acknowledgements}

The author is grateful to Professor Jorge Rodríguez-López for his guidance and insightful comments during the preparation of this article and acknowledges financial support from the Spanish Ministry of Science, Innovation and Universities (Grant FPU24/01545).

\end{document}